\documentclass{article}

\usepackage[T1]{fontenc}
\usepackage[utf8]{inputenc}
\usepackage{amsmath, amssymb, amsthm, latexsym, amsfonts, ifsym, bbm, enumerate, url, mathrsfs}

\newtheorem{theorem}{Theorem}[section]
\newtheorem{lemma}[theorem]{Lemma}
\newtheorem{defn}[theorem]{Definition}
\newtheorem{cor}[theorem]{Corollary}
 
\theoremstyle{definition} 
\newtheorem{notation}[theorem]{Notation}
\newtheorem{rmk}[theorem]{Remark}
\newtheorem{example}[theorem]{Example}

\newtheorem{quest}[theorem]{Question}

\numberwithin{equation}{section}

\title{Cutoff for the Transposition Walk on Permutations with One-Sided Restrictions}
\author{Olena Blumberg}
\date{}
\begin{document}
\maketitle

\begin{abstract}
This paper explores the mixing time of the random transposition walk on permutations with one-sided interval restrictions. In particular, we're interested in the notion of cutoff, a phenomenon which occurs when mixing occurs in a window of order smaller than the mixing time. One of the main tools of the paper is the diagonalization obtained by Hanlon \cite{HanlonPaper}; the use of the spectral information is inspired by the famous paper of Diaconis and Shahshahani on the mixing time of the random transposition walk on $S_n$ \cite{DiaconisandShah}. The diagonalization allows us to prove chi-squared cutoff for a broad class of one-sided restriction matrices. Furthermore, under an extra condition, the walk also undergoes total variation cutoff. Finally, a large collection of examples which undergo chi-squared cutoff but in which total variation mixing occurs substantially earlier and without cutoff is produced. These results resolve a conjecture of Diaconis and Hanlon from Section 5 of Hanlon's paper.  
\end{abstract}

\section{Introduction}

This paper studies the mixing time of the random transposition walk on permutations with one-sided interval restrictions. Intuitively, the mixing time of a random walk is the number of steps it takes a random walk to get close to random. Permutations with restricted positions are elements $\sigma$ of the symmetric group $S_n$ such that for each $i$, $\sigma(i)$ is only allowed to be in a particular subset of $\{1, 2, \dots, n\}$. These are specified by a $\{0,1\}$-matrix $M$, where $\sigma(i)$ is allowed to be equal to $j$ precisely if $M(i, j) = 1$. Such a matrix $M$ is called a restriction matrix. 

Let $S_M$ be the set of permutations corresponding to $M$. More concisely, 
\begin{equation}\label{permswithrestrictions}
S_M = \{\sigma \left| \right. M(i, \sigma(i)) = 1 \text{ for all } i\}
\end{equation}
Define $S(i) = \{j \left| \right. M(i, j) = 1\}$. Under this definition, $S(i)$ is the set of allowable values for $\sigma(i)$. If $S(i)$ is an interval for each $i$, then $M$ is called an interval restriction matrix. If $M$ is $n\times n$, and $S(i) = [a_i, n]$ for each $i$ for some $a_i$, then $M$ is a one-sided interval restriction matrix. For such matrices $M$, the permutations in $S_M$ correspond to rook placements on a Ferrers board, an object obtained by removing a Ferrers diagram from one of the corners of an $n \times n$ chessboard. These objects have an elegant combinatorial structure, originally studied by Foata, and Sch{\"u}tzenberger \cite{FoataPaper}, and then later by Goldman, Joichi and White \cite{GJW}. In particular, it is possible to perfectly sample elements of $S_M$. 

The random transposition walk on permutations with one-sided interval restrictions proceeds as follows. First, pick a number $i$ from $\{1, 2, \dots, n\}$ with your left hand, and an independent number $j$ from $\{1, 2, \dots, n\}$ with your right hand. Note that it is is possible to choose the same number. If the cards in positions $i$ and $j$ can be transposed while remaining in the set $S_M$, do so. If this is not possible, go back and pick a different pair $i'$ and $j'$ using the same procedure. Continue selecting pairs until arriving at one that can be exchanged without leaving $S_M$. Transposing this pair will then be the next step in the chain. 

This random walk was diagonalized by Hanlon \cite{HanlonPaper}, who derived all the eigenvalues and eigenvectors in terms of combinatorial objects described in Definition \ref{defbpartition}. This amazing result is an extension of the diagonalization of the random transposition walk on $S_n$, which uses the representation theory of the symmetric group (see James and Kerber \cite{JamesandKerber} for a treatment of the standard theory). The diagonalization of the walk on $S_M$ was possible despite the absence of a group structure. In this paper, the spectral information is used to obtain very sharp results on the mixing time of the chain under further restriction to `two-step' one-sided restriction matrices, where each $S(i)$ is either equal to $[1,n]$ or $[a, n]$ for some fixed $a$. This subclass of one-sided restriction matrices was selected because they generate a vertex transitive random walk and as such require only eigenvalue and not eigenvector information to examine the mixing time; this result is stated precisely in Corollary \ref{alleigsvertextransitive}.

Any discussion of mixing time requires a notion of distance: a random walk is close to random when the distance between the distribution of the random walk at time $t$ and the stationary distribution $\pi$ is small. A common distance is the total variation distance $\|\mu - \nu \|_{\mathrm{TV}}$, which is defined to be the $L^1$ distance between $\mu$ and $\nu$. This distance has the advantage of measuring an intuitive quantity: it is the maximum error that can be made by estimating the probability $\nu(A)$ with $\mu(A)$. However, this distance is not amenable to spectral analysis. As such, a second notion of distance is necessary. The appropriate distance is the chi-squared distance, which is defined to be the $L^2$ distance between $\frac{\mu}{\pi}$ and $1$ (see Equation \eqref{chisquaredefn}). It is easy to show that the chi-squared distance is an upper bound on the total variation distance. In many examples, the chi-squared and total variation mixing times match, making the chi-squared distance an excellent tool when spectral information is available. 

The advantage of having complete spectral information is that it can be used to obtain very precise bounds on chi-squared distance at any time $t$. In particular, it can be used to show that the random walk experiences cutoff: a phenomenon which occurs when the walk transitions from being barely mixed to being thoroughly mixed in a window much smaller than the mixing time (see Equation \eqref{cutoffdefn} for the precise definition). The notion of cutoff can be applied to any distance; accordingly, this paper will examine both total variation and chi-squared cutoff. The present approach is inspired by the results of Diaconis and Shahshahani, who used the eigenvalues of the random transposition walk on $S_n$ to show total variation cutoff of the walk at time $\frac{1}{2} n \log n$ with a window of size $n$ \cite{DiaconisandShah}. This paper makes use of the techniques of this proof, especially the more streamlined version presented by Diaconis in the book ``Group representations in probability and statistics'' \cite{DiaconisBook}. However, as will be seen below, the calculations in this paper are much more involved. 

The main results of this paper are as follows: for two-step restriction matrices, there is a broad class of matrices for which the random transposition walk experiences chi-squared cutoff. This result is stated precisely in Theorem \ref{chitheorem}. However, the total variation mixing time is more complicated: an extra condition is needed to ensure total variation cutoff (see Theorem \ref{TVlowerbound}), and these cases require an additional combinatorial argument. Furthermore, in Theorem \ref{TVfastmixingexample}, a class of examples is given for which total variation mixing occurs considerably earlier, and for which cutoff does not occur. In particular, these results resolve the conjecture stated as Theorem 5.6 in Hanlon's paper. Using the definitions above, the conjecture states that if 
\begin{equation*}
S(i) = \begin{cases} [1, n] & i \leq n^\alpha\\
                     [2, n] &    n^{\alpha} < i \leq n
        \end{cases} 
\end{equation*}
then the random walk has cutoff. (There's a slight typo in the paper: it states that cutoff occurs around time $\frac{\alpha}{2} n^{2-\alpha} \log n$, whereas it should actually be $\frac{\alpha}{4} n^{2-\alpha} \log n$.) This turns out to be partially correct: while the statement is true for chi-squared mixing time, it does not hold total variation mixing time. Indeed, the total variation mixing time occurs substantially earlier and without cutoff.

Therefore, the results in this paper are instructive in a number of ways. In a subset of cases, they show the utility of spectral information in showing cutoff with respect to both chi-squared and total variation distance. However, they also show that there is a broad class of random walks for which the spectral information only contains information about chi-squared distance, since total variation mixing occurs substantially earlier. Furthermore, in the latter case, it is demonstrated that the behavior of the chi-squared distance can be entirely different from the behavior of total variation distance, since in the examples provided, chi-squared distance undergoes cutoff while the total variation distance does not.

This paper also raises a number of fascinating questions about the random transposition walk on permutations with one-sided interval restriction. While this paper obtains strong results for a broad class of two-step restriction matrices, I am confident that some of the hypotheses in Theorems \ref{chitheorem} and \ref{TVlowerbound} are extraneous. Furthermore, while it was convenient to restrict attention to the two-step case in order to simplify calculations by using only eigenvalue information, Hanlon's paper also derives the eigenvectors of the random walk. Therefore, while the calculations would likely be more challenging, it should be possible to apply the ideas below to a broader class of one-sided interval restriction matrices. It would also be very interesting to find conditions on one-sided restriction matrices which would guarantee that the chi-squared and total variation mixing time are of the same order. For a full discussion of the many questions which arise naturally from this paper, see Section \ref{sectionquestionstoponder}. 

\section{Definitions and Setup}

There are a number of definitions needed before the results can be fully stated. Let $\mu$ and $\nu$ are two probability distributions on a finite state space $\Omega$. Then, the total variation distance between $\mu$ and $\nu$ is defined as 
\begin{equation*}
\left\| \mu - \nu \right\|_{TV} = \frac{1}{2}\sum_{x\in \Omega} | \mu(x) - \nu(x)|
\end{equation*}
For a Markov chain with transition probabilities $P(x, y)$ and stationary distribution $\pi$, define $d(t) =  \left\| P^t(x, \cdot) - \pi\right\|_{TV}$ to be the distance from stationarity at time $t$. Then, the mixing time is defined as 
\begin{equation*}
\tau_{\mathrm{mix}}(\epsilon) = \min\left\{t \left| \right. d(t) \leq \epsilon\right\}
\end{equation*}
Conventionally, $\tau_{\mathrm{mix}}$ is chosen to be $\tau_{\mathrm{mix}}(1/4)$. 

As noted above, this paper also uses a second notion of distance. Define the chi-squared distance between distributions $\mu$ and $\pi$ with respect to $\pi$ to be 
\begin{equation}\label{chisquaredefn} 
\left\| \mu - \pi \right\|_{2, \pi}  = \sqrt{\sum_{x} \left( \frac{\mu(x)}{\pi(x)} - 1 \right)^2 \pi(x) }
\end{equation} 
The chi-squared mixing time can be defined in a way which is entirely analogous to the definition of the total variation mixing time. 

As should be readily apparent, the chi-squared distance is the $L^2(\pi)$ distance between $\frac{\mu}{\pi}$ and $1$. Furthermore, it is easy to see that $\left\| \mu - \pi \right\|_{TV}$ is precisely half the $L^1(\pi)$ distance between $\frac{\mu}{\pi}$ and $1$. This means that chi-squared distance is an upper bound for twice the total variation distance. Furthermore, chi-squared distance turns out to precisely computable if all the eigenvalues and eigenvectors of the matrix $P$ are known. The following two results are standard, and follow from Lemma 12.16 in \cite{YuvalBook}. 

\begin{theorem}\label{alleigs}
Let $P$ be the transition matrix of an irreducible reversible Markov chain on $\Omega$ with eigenvalues $1 = \beta_0 > \beta_1 \geq \cdots \geq \beta_{|\Omega|-1} \geq -1$, and with a corresponding basis of orthonormal eigenvectors $\{v_j\}$. Then,
\begin{equation*}
\left\| P^t(x, \cdot) - \pi \right\|_{TV} \leq \frac{1}{2} \left\| P^t(x, \cdot) - \pi \right\|_{2, \pi} = \frac{1}{2}\sqrt{\sum_{i =1}^{|\Omega|-1} v_i(x)^2 \beta_i^{2t}}
\end{equation*} 
\end{theorem}

The above theorem requires all the information about the eigenvalues and eigenvectors to compute the chi-squared distance. However, given a sufficient degree of symmetry, eigenvalues suffice. Say that a Markov chain on a state space $\Omega$ is {\it vertex transitive} if for every $x, y \in \Omega$, there exists a bijection $f: \Omega \rightarrow \Omega$ such that $f(x) = y$, and $f$ preserves the random walk. In this case, the following easy corollary can be proved. 

\begin{cor}\label{alleigsvertextransitive}
If the assumptions from Theorem \ref{alleigs} hold, and the Markov chain is also vertex transitive, then
\begin{equation*}
\left\| P^t(x, \cdot) - \pi \right\|_{TV} \leq \frac{1}{2} \left\| P^t(x, \cdot) - \pi \right\|_{2, \pi} = \frac{1}{2}\sqrt{\sum_{i =1}^{|\Omega|-1} \beta_i^{2t}}
\end{equation*}
\end{cor} 

Finally, it is necessary to define the cutoff phenomenon rigorously. This definition can apply to any notion of mixing time; in this paper, it will be applied both to total variation and chi-squared mixing times. Consider a family of Markov chains $(X^n_t)_{t\geq 0}$ with mixing times $\tau_{\mathrm{mix}}^{n}(\epsilon)$. Then, cutoff is present if for all $\epsilon > 0$,
\begin{equation}\label{cutoffdefn} 
\lim_{n\rightarrow \infty} \frac{\tau_{\mathrm{mix}}^{n}(\epsilon)}{\tau_{\mathrm{mix}}^{n}(1 - \epsilon)} = 1 
\end{equation}
This is clearly equivalent to saying that for sufficiently large $n$, mixing occurs in an arbitrarily small window relative to the size of $\tau^n_{\mathrm{mix}}$.

Now proceed to definitions specific to one-sided interval restriction matrices. Let $M$ be such a matrix; then for each $i$ $S(i)$ must be equal to $[b_i, n]$ for some $b_i$. To simplify notation, make the following definition. 

\begin{defn}
Let $n$ be an integer, and let $\vec{b} = (b_1, b_2, \cdots, b_n)$ be a vector of elements of $\{1, 2, \dots, n\}$.  Define $M(\vec{b})$ to be the $n \times n$ interval restriction matrix satisfying $S(i) = [b_i, n]$ for all $i.$ 
\end{defn}
Here is an example: if $n=4$ and $\vec{b} = (1,1,2,3)$, then  
\begin{equation*}
M(\vec{b})=
\begin{bmatrix}
1 & 1 & 1 & 1 \\
1 & 1 & 1 & 1 \\
0 & 1 & 1 & 1 \\
0 & 0 & 1 & 1
\end{bmatrix}
\end{equation*}

\begin{rmk}
The notation here is different from Hanlon's -- he doesn't explicitly talk about restriction matrices and uses the notation $R_n(\vec{b})$ for the set we have called $S_{M(\vec{b})}$. 
\end{rmk}

\begin{lemma}\label{permutingrows}
Let $M$ be a restriction matrix, and let $N$ be $M$ with its rows permuted. Then there is an isomorphism between the random transposition walk on $N$ and the random transposition walk on $M$.
\end{lemma}
\begin{proof}[\bf Proof:] 
This result is straightforward (and is simply assumed in Hanlon \cite{HanlonPaper}.) For a full discussion, see Lemma 3.1 in my thesis \cite{OlenaThesis}. 
\end{proof}

Lemma \ref{permutingrows} shows that permuting the rows of $M$ doesn't change the underlying graph. Therefore, from now on assume that the vector $\vec{b}$ satisfies 
\begin{equation}\label{orderingbis}
b_1 \leq b_2 \leq \cdots \leq b_n
\end{equation}
With the above assumption, Proposition 2.1 (d) in Hanlon shows that:

\begin{lemma}\label{regulargraph}
If $b_i \leq i$ for all $i$ then the graph on $S_{M(\vec{b})}$ induced by the random transposition walk is regular, with degree 
\begin{equation}\label{Deltavalue}
\Delta = \sum_{i=1}^n (i - b_i)
\end{equation}
If $b_i > i$ for some $i$, the set $S_{M(\vec{b})}$ is empty. 
\end{lemma}

With the random transposition walk as defined in the introduction, make the following definition:

\begin{defn}
Given a vector $\vec{b} = (b_1, b_2, \dots, b_n)$, define $U(\vec{b})$ to be the adjacency matrix of $S_{M(\vec{b})}$ under the random transposition walk. 
\end{defn}

With $\Delta$ from Equation \ref{Deltavalue}, if $P$ is the transition matrix for the random walk, then for $\tau$ and $\sigma$ that differ by a transposition $(i,j)$,
\begin{equation*}
P(\sigma, \tau) = \frac{2}{n+2\Delta}
\end{equation*}
since $i$ can be chosen with the right hand and $j$ with the left hand, or vice versa. Furthermore, 
\begin{equation*}
P(\sigma, \sigma) = \frac{n}{n+2\Delta}
\end{equation*}
since the same $i$ can be chosen with both hands. This shows that
\begin{equation}\label{Ptransitionmatrix}
P = \frac{1}{n+ 2\Delta}(nI + 2 U)
\end{equation} 
where $I$ is the identity matrix. This demonstrates that finding the spectral information for $P$ reduces to diagonalizing $U$. Correspondingly, many of the theorems in this paper  will be stated in terms of $U(\vec{b})$. 

As noted earlier, attention here is restricted to a subclass of one-sided restriction matrices: 

\begin{defn}\label{boffg}
Let $f(n), g(n)$ be two functions from $\mathbb{Z}^+$ to $\mathbb{Z}^+$. Define $\vec{b}_n(f, g)$ to be the vector $(b_1, b_2, \dots, b_n)$ such that 
\begin{equation*}
b_i = \begin{cases}
1 & i \leq f(n) \\
g(n)+1 &  f(n) < i \leq n
\end{cases}
\end{equation*}
That is, $\vec{b}_n(f, g) = (1, 1, \dots, 1, g(n)+1, \dots, g(n)+1)$ where the number of $1$s in the beginning of the vector is $f(n)$. 
\end{defn} 

\begin{rmk}\label{fgreatherthang} From Lemma \ref{regulargraph}, $b_i \leq i$ for all $i$ is necessary for $S_{M(\vec{b})}$ not to be empty. This condition at $i = g(n)+1$ above forces
\begin{equation}\label{fgreaterthangeq}
f(n) \geq g(n)
\end{equation}
to a have a non-empty walk. Thus, from here on assume the above inequality for functions $f$ and $g$ for all $n$.
\end{rmk}

Now, let $\vec{b} = \vec{b}_n(f, g)$ as defined above, and consider the matrix $M(\vec{b})$. By the above definitions, $S[i]$ is $[1, n]$ if $i$ is between $1$ and $f(n)$, while $S[i]$ is $[g(n)+1, n]$ for $i$ between $f(n)+1$ and $n$. This means that $\sigma \in S_{M(\vec{b})}$ has no restrictions on the first $f(n)$ rows, and must be at least $g(n)+1$ on rows that are at least $f(n)+1$. Alternatively, on the first $g(n)$ columns, $\sigma$ is only allowed to take values up to $f(n)$. For example, if $n = 5$, $f(5) = 3$ and $g(5) = 2$, then
\begin{equation}\label{boffgexample}
\vec{b} = \vec{b}_5(f, g) = (1, 1, 1, 3, 3)
\end{equation}
and 
\begin{equation*}
M(\vec{b}) = 
\begin{bmatrix}
1 & 1 & 1 & 1 & 1\\
1 & 1 & 1 & 1 & 1\\
1 & 1 & 1 & 1 & 1\\
0 & 0 & 1 & 1 & 1\\
0 & 0 & 1 & 1 & 1\\
\end{bmatrix}
\end{equation*}

The next lemma calculates the degree $\Delta$ of a vertex in $S_M$ for above matrices $M$. 

\begin{lemma}\label{deltafg}
Let $\vec{b}_n(f, g)$ be defined as in Definition \ref{boffg} above, and let $M_n = M(\vec{b}_n(f, g))$. Then, if $\Delta$ is defined as in Equation \eqref{Deltavalue} above, 
\begin{equation*}
\Delta = \frac{n^2 - n - 2ng(n) + f(n)g(n)}{2}
\end{equation*}
\end{lemma}
\begin{proof}[\bf Proof:]
From the definition of $b_i$,
\begin{align*}
\Delta &= \sum_{i=1}^n  (i - b_i) = \frac{n(n+1)}{2} - \sum_{i=1}^n b_i \\
       &= \frac{n(n+1)}{2} - f(n)\cdot 1 - (n-f(n))\cdot (g(n)+1) \\
       &= \frac{n^2 - n - 2ng(n) + f(n)g(n)}{2}
\end{align*}
as desired.
\end{proof}

\begin{notation}
From now on, when $f$ and $g$ are implied, the convention
\begin{equation}\label{Mconvention}
M_n = M(\vec{b}_n(f, g))
\end{equation}
is used to simplify notation. 
\end{notation}

The first theorem below shows that, given certain assumptions on $f(n)$ and $g(n)$, the walk always achieve cutoff in chi-squared distance. 

\begin{theorem}\label{chitheorem}
Let $\vec{b}_n = \vec{b}_n(f, g)$ for some functions $f$ and $g$ which satisfy $f(n) \geq g(n)$ for all $n$, and which also satisfy 
\begin{equation*}
\lim_{n\rightarrow \infty}\frac{f(n)}{n} = 0 \textnormal{ and } \lim_{n\rightarrow \infty} f(n) = \infty
\end{equation*}
and consider the random transposition walk on $S_{M_n}$. Let $\Delta$ denote the degree of the graph induced on $S_M$; hence, from Equation \ref{deltafg} above, $n+2\Delta = n^2 - 2n g(n) + f(n)g(n)$. Then, there is chi-squared cutoff around 
\begin{equation*}
\frac{(n+2\Delta)(\log f(n) + \log g(n))}{4f(n)}
\end{equation*}
with a window of size $\frac{n+2\Delta}{4f(n)}$. More precisely, 
\begin{enumerate}
\item 
If 
\begin{equation}\label{tchiupperbound}
t = \frac{(n+2\Delta)(\log f(n) + \log g(n))}{4f(n)} + c \frac{n+2\Delta}{4f(n)}
\end{equation}
then 
\begin{equation*}
\left\| P^t(x, \cdot) - \pi \right\|_{2, \pi} \leq 4e^{-\frac{c}{2}}
\end{equation*}
for $c > 10$ and $n$ sufficiently large. 

\item  Furthermore, if $c > 0$ and 
\begin{equation}\label{tchilowerbound}
t = \frac{(n+2\Delta)(\log f(n) + \log g(n))}{4f(n)} - c \frac{n + 2\Delta}{4f(n)}
\end{equation}
then 
\begin{equation*}
\left\| P^t(x, \cdot) - \pi \right\|_{2, \pi} \geq \frac{1}{2} e^{c/2}
\end{equation*}
for $n$ sufficiently large.  
\end{enumerate}
\end{theorem} 

As noted earlier, there are also interesting results for the more intuitive total variation distance. The following theorem states assumptions on $f$ and $g$ which imply total variation cutoff for the walk, while the theorem after it presents a class of examples for which total variation mixing occurs significantly before chi-squared mixing, and in which total variation cutoff does not occur. 

\begin{theorem}\label{TVlowerbound}
Assume that $f$ and $g$ satisfy the hypotheses for Theorem \ref{chitheorem} above, and furthermore, that $\liminf_{n \rightarrow \infty} \frac{g(n)}{f(n)} > r > 0$. Let $M_n = M(\vec{b}_n(f, g))$ as in Equation \eqref{Mconvention} above. Then, if 
\begin{equation}\label{mTVlowerbound}
 t = \frac{(n+2\Delta) (\log g(n) + \log f(n))}{4f(n)} - c \frac{n+2\Delta}{4f(n)}
\end{equation}
then the random transposition walk on $S_{M_n}$ satisfies
\begin{equation*}
\liminf_{n\rightarrow \infty} \left\| P^t(x, \cdot) - \pi \right\|_{TV} \geq \frac{1}{e} - e^{-re^c}
\end{equation*}
and thus the walk hasn't mixed by time $t$ for sufficiently large $c$. 
\end{theorem}

\begin{theorem}\label{TVfastmixingexample}
Let $g(n) = 1$, and assume that 
\begin{equation*}
\lim_{n\rightarrow \infty} f(n) = \infty \text{ and } f(n) \leq \frac{n}{5\log n} \text{ for sufficiently large $n$.}
\end{equation*}
Then, the random transposition walk on $S_{M_n}$ mixes in total variation distance in order $\frac{n+2\Delta}{f(n)}$ time. Furthermore, the walk does not have total variation cutoff. 
\end{theorem}
\begin{rmk}
Hanlon's conjecture in his Theorem 5.6 concerns the case where $g(n)= 1$ and $f(n) = n^\alpha$. As such, it is one of the many examples covered by the theorems above. After a little simplification, Theorem \ref{chitheorem} states that this walk has cutoff around time $\frac{\alpha}{4}n^{2- \alpha} \log(n)$ with a window of size $n^{2-\alpha}$. It is also clear that this case is one of the examples covered by Theorem \ref{TVfastmixingexample}. Therefore, total variation mixing occurs in order $n^{2-\alpha}$ time, and occurs without cutoff. Thus, the above results resolve the conjecture. 
\end{rmk} 

The next section gives heuristics and proves Theorem \ref{TVlowerbound}, while Section \ref{fastmixingchapter} proves Theorem \ref{TVfastmixingexample}, and also proves that the walk is vertex transitive. The chi-squared results are then proved. Section \ref{hanlonpreviouswork} reviews the necessary background material, restates Hanlon's eigenvalue results and provides examples of how they apply for the specific matrices $M_n = M(\vec{b}_n(f, g))$ under consideration. Section \ref{boundingtheeigenvalues} provides bounds on the eigenvalues of the transition matrix $P$, which will be used in the following section to do lead term analysis and prove the lower bound in Theorem \ref{chitheorem}. Finally, Sections \ref{dimensionofeigenspacecalc} and \ref{chisquaredupperchapter} prove the upper bound in Theorem \ref{chitheorem}. 

\section{Heuristics and Total Variation Lower Bound}\label{sectheuristics}

This section gives heuristics and proves Theorem \ref{TVlowerbound}. First, here is an argument explaining why a total variation cutoff is expected around 
\begin{equation*}
 t = \frac{(n+2\Delta)(\log f(n) + \log g(n))}{4f(n)}
\end{equation*}
if $f(n)$ and $g(n)$ are of the same order, as the conjunction of Theorems \ref{chitheorem} and \ref{TVlowerbound} implies. 

Consider the instructive case $f(n)= g(n)$. For example, if $n= 5$ and $f(n) = g(n) = 2$, the restriction matrix is
\begin{equation*}
M = \begin{bmatrix} 
1 & 1 & 1 & 1 & 1 \\
1 & 1 & 1 & 1 & 1 \\
0 & 0 & 1 & 1 & 1 \\
0 & 0 & 1 & 1 & 1 \\
0 & 0 & 1 & 1 & 1 \\
\end{bmatrix}
\end{equation*}
In this case, $\sigma \in S_M$ can take values up to $f(n)$ on the first $f(n)$ columns. This makes it easy to see that $\sigma\in S_m$ can be represented as a pair of permutations, one on $\{1, 2, \dots, f(n)\}$ and the other on $\{f(n) +1, f(n)+2, \dots, n\}$. Thus, 
\begin{equation*}
S_M \cong S_{f(n)} \times S_{n - f(n)}
\end{equation*}

Now consider how fast the walk mixes. It has been shown by Diaconis and Shahshahani that the transposition walk on $S_k$ has cutoff at time $\frac{1}{2} k \log k$ with a window of size $k$ \cite{DiaconisandShah}. At each step, the present walk uses either a transposition in $S_{f(n)}$ (with probability $\frac{f(n)^2}{n+2\Delta}$) or a transposition in $S_{n - f(n)}$ (with probability $\frac{(n-f(n))^2}{n+2\Delta}$). This is equivalent to two separate walks: the walk on $S_{f(n)}$, slowed down by a factor of $\frac{n+2\Delta}{f(n)^2}$, and the walk on $S_{n -f(n)}$, slowed down by a factor of $\frac{n+2\Delta}{(n - f(n))^2}$. Thus, the walk on $S_{f(n)}$ has cutoff around time
\begin{equation*}
t_1 = \frac{1}{2} f(n) \log f(n) \cdot \frac{n+2\Delta}{f(n)^2} = \frac{(n+2\Delta)\log f(n)}{2 f(n)}
\end{equation*}
with a window of size $f(n) \cdot \frac{n+2\Delta}{f(n)^2} = \frac{n+2\Delta}{f(n)}$, while the walk on $S_{n-f(n)}$ analogously has cutoff around time 
\begin{equation*}
 t_2 = \frac{(n+2\Delta)\log(n - f(n))}{2(n - f(n))}
\end{equation*}
with a window of size $\frac{n+2\Delta}{n - f(n)}$. It is easy to see (and is  proved below in Lemma \ref{boundcompare}), that if $f(n)$ is comparitively small with respect to $n$, then $t_1$ is the larger of the two times. Since in this case, $f(n)= g(n)$, $t_1$ gives precisely the answer of Theorems \ref{chitheorem} and \ref{TVlowerbound}. The above analysis shows that the first $g(n)$ columns are the limiting component of the walk. Theorem \ref{chitheorem} assumes that $\frac{f(n)}{n} \rightarrow 0$ to simplify computations.

\begin{rmk} \label{whyfgoestoinfinity}
This discussion shows why $f(n) \rightarrow \infty$ is assumed for the main theorems -- in order to reasonably talk about asymptotics, the limiting term should at least go to $\infty$. 
\end{rmk}

Before proving Theorem \ref{TVlowerbound}, a supporting lemma calculating the size of $S_{M(\vec{a})}$ is needed.

\begin{lemma}\label{sizelemma}
If $\vec{a} = (1, 1, \dots, 1, y+1, \dots, y+1)$ is a vector of length $n$, where the number of initial $1$s is precisely $x$, and $x \geq y$, then  
\begin{equation*}
\left|S_{M(\vec{a})}\right| = \frac{x!(n-y)!}{(x-y)!}
\end{equation*}
\end{lemma}
\begin{proof}[\bf Proof:]
Count the number of elements in $S_{M(\vec{a})}$ column by column. How many choices are there for $\alpha^{-1}(1)$? Since $\alpha(i)$ is only allowed to be $1$ for $i\leq x$, there are $x$ choices for the first column. Similarly, there are $(x-1)$ choices for $\alpha^{-1}(2)$, and continuing, there are $(x-i+1)$ choices for $\alpha^{-1}(i)$ for $i\leq y$. Thus, the number of choices for the tuple $(\alpha^{-1}(1), \alpha^{-1}(2), \dots, \alpha^{-1}(y))$ is precisely
\begin{equation*}
x\cdot(x-1) \cdots (x-y+1) = \frac{x!}{(x-y)!} 
\end{equation*}

Now, consider the number of choices for $\alpha^{-1}(y+1)$. Since $\alpha(i)$ is allowed to be equal to $y+1$ for every single $i\leq n$, and furthermore, $\alpha^{-1}(1), \cdots, \alpha^{-1}(y)$ are already specifiied, there are  precisely $n - y$ choices for $\alpha^{-1}(y+1)$, $n-y-1$ choices for $\alpha^{-1}(y+2)$, etc. Thus, the total number of choices for these is precisely $(n-y)\cdot(n-y-1) \cdots 2 \cdot 1 = (n-y)!$. 
Multiplying these together gives
\begin{equation*}
\left| S_{M(\vec{a})}\right| = \frac{x!}{(x-y)!} (n-y)! = \frac{x!(n-y)!}{(x-y)!}
\end{equation*}
as required. 
\end{proof}

Theorem \ref{TVlowerbound} gives a lower bound on the mixing time. Here, an explicit set $A$ is found such that for $t$ in the above theorem, $| P^t( id, A) - \pi(A) |$ is large. In their paper, Diaconis and Shahshahani use the set of permutations with at least $1$ fixed point \cite{DiaconisandShah}. The heuristics above suggest trying a modification: clearly, if $f(n) = g(n)$, the fixed points in the top $f(n) \times f(n)$ square of $M_n$ could be used. In general, the number of fixed  points in the first $g(n)$ columns works. 

Before the proof, some notation mainly used for the next two sections is needed. Denote the Markov chain by $(X_k)_{k=0}^\infty$, where 
\begin{equation*}
X_k = (X_k(1), X_k(2), \dots, X_k(n))
\end{equation*}
and let the random row transposition used at time $k$ be $r_k$. Clearly, $X_k = X_0 r_1 r_2 \dots r_k$. 

\begin{proof}[\bf Proof of Theorem \ref{TVlowerbound}:]
It must be shown that at time 
\begin{equation}\label{TVlowerboundrestated}
 t = \frac{(n+2\Delta) (\log g(n) + \log f(n))}{4f(n)} - c \frac{n+2\Delta}{4f(n)}
\end{equation}
the walk on $S_{M_n}$ has not yet mixed, where $M_n = M(\vec{b}_n(f,g))$. By assumption, 
\begin{equation*}
\liminf_{n \rightarrow \infty} \frac{g(n)}{f(n)} = r > 0.
\end{equation*}
This implies that $\limsup_{n\rightarrow \infty} (\log f(n) - \log g(n)) < - \log r$. Thus, for sufficiently large $n$, 
\begin{equation*}
t < \frac{(n+2\Delta) \log g(n)}{2f(n)} - (c+\log r) \frac{n+2\Delta}{4f(n)}.
\end{equation*}
Hence it suffices to show that the walk has not mixed at the time on the right-hand side above. Call this time $t'$. 

Define $A_n$ as 
\begin{equation*}
A_n = \left\{ \sigma \in S_{M_n} \text{ such that } \sigma(i) = i \text{ for at least one } i \leq g(n)\right\}
\end{equation*}
Call $i \leq g(n)$ such that $\sigma(i) = i$ a {\it small fixed point}. The first step calculates an upper bound for $\pi(A_n)$ for the uniform distribution $\pi$ on $S_{M_n}$. This is done by a standard inclusion-exclusion argument. Consider the number of $\sigma \in S_{M_n}$ such that $\sigma(i) = i$ for a specified value of $i\leq g(n)$. In the language of permutation matrices, to pick such a $\sigma$, mark the $1$ in position $(i, i)$, cross out the $i$th row and $i$th column, and select the values of $\sigma$ on the remaining rows and columns. It is easy to check that this corresponds to picking an element in $S_{M(\vec{a})}$, for $\vec{a} = (1, 1, \dots, 1, g(n), \dots, g(n))$ where $\vec{a}$ starts with $f(n)-1$ ones and the vector is of length $n-1$. Thus, from Lemma \ref{sizelemma} above, 
\begin{equation*}
\left| \{ \sigma \in S_{M_n} \left| \right. \sigma(i) = i \} \right| = \left| S_{M(\vec{a})} \right| = \frac{(f(n)-1)!(n - g(n))!}{(f(n) - g(n))!}
\end{equation*}
From the same lemma, $S_{M_n} = \frac{f(n)!(n -g(n))!}{(f(n) - g(n))!}$. Thus,
\begin{equation*}
\left| \{ \sigma \in S_{M_n} \left| \right. \sigma(i) = i \} \right|  = \frac{\left| S_{M_n} \right|}{f(n)}.
\end{equation*}
In a similar way, given $k$ distinct values of $i_1, i_2, \dots, i_k$ in $[1, g(n)]$, choose $\sigma$ satisfying $\sigma(i_j) = i_j$ for all $j$ by crossing out all the rows $i_j$ and columns $i_j$. This would correspond to the vector $\vec{a}' = (1, 1, \dots, 1, g(n)+1 - k, \dots, g(n)+1-k)$ of length $n-j$, and thus 
\begin{align*}
\left| \{ \sigma \in S_{M_n} \left| \right. \sigma(i_1) = i_1, \dots, \sigma(i_k) = i_k \} \right|  &= \frac{(f(n) - k)!(n - g(n))!}{(f(n) - g(n))!} \\
 &= \left| S_{M_n} \right| \frac{f(n)!}{(f(n) - k)!}
\end{align*}
Picking an element $\sigma$ that has precisely $k$ small fixed points involves choosing the values $i_1, i_2, \dots, i_k$ above. Thus, 
\begin{align*}
\left| \{ \sigma \in S_{M_n} \left| \right. \text{$\sigma$ has $k$ small fixed points} \}\right| &= {g(n) \choose k} \left| S_{M_n} \right| \frac{f(n)!}{(f(n) - k)!}\\
& = \left| S_{M_n} \right| \frac{g(n) \cdots (g(n) - k+1)}{k! f(n) \cdots (f(n) - k+1)}
\end{align*}
Now, the standard inclusion-exclusion argument gives
\begin{equation*}
|A_n| =  \left| S_{M_n} \right| \left( \frac{ g(n)}{f(n)}- \frac{g(n)(g(n)-1)}{2f(n)(f(n) - 1)} +\frac{g(n)(g(n)-1)(g(n)-2)}{6f(n)(f(n) - 1)(f(n)-2)} - \dots \right) 
\end{equation*}
From $f(n) \rightarrow \infty$, and $\liminf \frac{g(n)}{f(n)} >0$, $g(n)\rightarrow \infty$. Thus, the above sum becomes arbitrarily well-approximated by
\begin{equation*}
 \left| S_{M_n} \right| \left( \frac{g(n)}{f(n)}- \frac{g(n)^2}{2!f(n)^2} + \frac{g(n)^3}{3!f(n)^3} - \dots\right) =  \left| S_{M_n} \right| (1- e^{-g(n)/f(n)})
\end{equation*}
and thus $\pi(A_n)$ is arbitrarily well-approximated by $1 - e^{-g(n)/f(n)}$. Therefore, noting that $1 - e^{-x}$ is an increasing function, and that $\frac{g(n)}{f(n)} \leq 1$,
\begin{equation*}
\limsup \pi(A_n) \leq 1 - \frac{1}{e}.
\end{equation*}

Now, consider the probability of $A_n$ after $t'$ steps of the random walk, started at the identity. Define $F_k$ to be the number of small fixed points of our random walk at time $k$: that is, the number of rows $i\leq g(n)$ such that $X_k(i) = i$. Consider the distribution of $F_k$. Note that since the walk starts at the identity, any $i \leq g(n)$ that has not yet been tranposed with anything is in $F_t$. Furthermore, if $X_k(i) \leq g(n)$, row $i$ can only be transposed with row $j$ if $j\leq f(n)$ -- that is, a row corresponding to one of the first $g(n)$ columns can only be transposed with one of the first $f(n)$ rows. Otherwise, $X_{k+1}(j) = X_k(i) \leq g(n)$, which is not allowed for $j > f(n)$.

From the arguments above, $F_0 = \{1, 2, \dots, g(n)\}$, and furthermore, row $i$ can leave the set $F_{k-1}$ only by being transposed with a row that is at most $f(n)$. Consider two cases -- first, the next step could transpose two rows which are both at most $g(n)$. This has probability
\begin{align*}
\mathbb{P} \left( r_k = (ij) \left| \right. i, j \leq g(n)\right) = \frac{g(n)^2}{n+2\Delta}.
\end{align*}
Secondly, the next step could transpose a pair of rows one of which is below $g(n)$ and the other one is between $g(n)$ and $f(n)$. This has probability
\begin{align*}
\mathbb{P} \left( r_k = (ij) \left| \right. i \leq g(n) < j \leq f(n) \text{ or vice versa}\right) = \frac{2g(n)(f(n) - g(n))}{n+2\Delta}.
\end{align*}

The argument proceeds by estimating the probability of having transposed each element from $\{1, 2, \dots, g(n)\}$ at time $t'$. Note that if two rows below $g(n)$ are transposed at time $k$, then they both leave the set $F_{k-1}$, whereas the second case above corresponds to only one row leaving the set. This rephrases the question as the following coupon collectors problem: at each step, collect two coupons with probability $\frac{g(n)^2}{n+2\Delta}$ and one coupon with probability $\frac{2g(n)(f(n) - g(n))}{n+2\Delta}$. What is the probability of not collecting every coupon from $\{1, 2, \dots, g(n)\}$ by time $t'$? 

Begin by calculating the mean number of coupons collected at each step, counting each coupon however many times it's collected. This number is clearly  
\begin{align*}
\mathbb{E} \left( \text{coupons collected in one step} \right) &= 2 \cdot \frac{g(n)^2}{n+2\Delta} + 1 \cdot \frac{2g(n)(f(n) - g(n))}{n+2\Delta} \\
&=\frac{2f(n)g(n)}{n+2\Delta} 
\end{align*}
Similarly, the variance of the number of coupons collected in one step is bounded above by $\frac{4f(n) g(n)}{n+2\Delta}$. Thus, after $t'$ total steps of the walk, the total number of coupons collected is concentrated around 
\begin{align*}
t' \cdot \frac{2 f(n) g(n)}{n+2\Delta} &= \left(\frac{(n+2\Delta) \log g(n)}{2f(n)} - (c+\log r)\frac{n+2\Delta}{2f(n)}\right) \cdot \frac{2 f(n) g(n)}{n+2\Delta}\\
&= g(n) \log g(n) - (c+\log r) g(n)
\end{align*}
with a window of at most $\sqrt{2 g(n) \log g(n)}$. Since $g(n)$ approaches $\infty$, and since $\sqrt{2g(n) \log g(n)}$ is $o(g(n))$, using standard coupon collector arguments \cite{Feller1} we can conclude that 
\begin{equation*}
P^{t'}(id, A_n)  \rightarrow  1 -  e^{-e^{c+\log r}} = 1 - e^{-re^c} 
\end{equation*}
as $n\rightarrow \infty$.

Therefore, 
\begin{align*}
\liminf_{n\rightarrow \infty} \left\| P^{t'} (id, \cdot)  -\pi \right\|_{TV} &\geq \liminf_{n\rightarrow \infty} \left( P^{t'}(id, A_n) - \pi(A_n)\right)\\
  &\geq \left( 1 -  e^{-re^{c}}\right) - \left( 1 -  \frac{1}{e}\right) \\
  &= \frac{1}{e} - e^{-re^{c}},
\end{align*}
the desired inequality for $t'$. As noted earlier, since for sufficiently large $n$, $t\leq t'$, and since distance to stationarity is non-decreasing, 
\begin{equation*}
\liminf_{n\rightarrow \infty} \left\| P^t (id, \cdot)  -\pi \right\|_{TV}  \geq \frac{1}{e} - e^{-re^{c}}
\end{equation*}
and so the walk hasn't mixed by time $t$, as required. 
\end{proof}

\section{Vertex Transitivity and Example of Fast Mixing} \label{fastmixingchapter}

This section contains the proof of some of the total variation results, in particular Theorem \ref{TVfastmixingexample}. Here, $g(n) = 1$, which means that $\vec{b}_n(f, g) = (1,  \dots, 1, 2, \dots, 2)$. Hence $\sigma \in S_{M_n}$ has the following restrictions: it's allowed to be at most $f(n)$ on column $1$, and has no restrictions at all on the remaining columns. It is easy to calculate that the probability of the first column being tranposed in a particular step is 
\begin{equation*}
P( \text{Column $1$ tranposed at step $k$}) = \frac{2f(n)-1}{n+2\Delta} \approx \frac{2f(n)}{n^2}
\end{equation*}
using the value for $\Delta$ derived in Lemma \ref{deltafg}. 

Now, since $f(n) \leq \frac{n}{5\log n}$ for sufficiently large $n$, the first column gets tranposed at most every $\frac{5}{2}n \log n$ steps or so. Furthermore, the remaining $n-1$ columns have no restrictions, and hence the walk without the first column is just the walk on $S_{n-1}$. Since this walk has cutoff at time $\frac{1}{2}(n-1) \log (n-1)$, these $n-1$ columns should be thoroughly mixed by the time the first column is used at all. Therefore, the walk should be mixed as soon as the first column is used. Furthermore, since mixing is driven by one column, a cutoff is not expected. 

In order to simplify calculations, it is first shown that for any $\vec{b} = \vec{b}_n(f, g)$, $S_{M(\vec{b})}$ is vertex transitive. This will simplify the  proof of Theorem \ref{TVfastmixingexample}, as the walk may be started at the identity. Furthermore, this will also be useful later for spectral analysis, as Corollary \eqref{alleigsvertextransitive} will be used. 

\begin{lemma}\label{vertextrans}
For $M_n = M(\vec{b}_n(f, g))$, the graph induced on $S_{M_n}$ by the random transposition walk is vertex transitive.
\end{lemma} 
\noindent \textbf{Note:} This paper uses a different convention from Hanlon for permutation multiplication: if $\alpha$ and $\beta$ are permutations, multiplication is treated as function composition. That is, 
\begin{equation*}
(\alpha \beta)(i) = \alpha(\beta(i))
\end{equation*}

\begin{proof}[\bf Proof:] Let $\pi$ and $\sigma$ be elements of 
$S_{M_n}$. A graph isomorphism $\phi: S_{M_n} \rightarrow S_{M_n}$ is found such that
\begin{equation*}
\phi(\pi) = \sigma.
\end{equation*}
Only graph isomorphisms of the form 
\begin{equation}\label{phidef}
\phi(\tau) = \alpha \tau \beta
\end{equation}
are considered, where $\alpha$ is a permutation in $S_n$ that only acts non-trivially on the elements $\{g(n)+1, g(n)+2, \dots, n\}$  and $\beta$ is an elements of $S_n$ that only acts non-trivially on the elements $\{1,2, \dots, f(n)\}$. The proof consists of the following steps. First, it is shown that any $\phi$ as defined in \eqref{phidef} maps $S_{M_n}$ to itself. Secondly, it is shown that any such $\phi$ is an isomorphism. Finally, $\alpha$ and $\beta$ are found such that the $\phi$ defined above satisfies 
\begin{equation*}
\phi(\pi) = \sigma
\end{equation*}
\textbf{Step 1.} Note that visually, $\alpha$ rotates the last $n- g(n)$ columns of the restriction matrix, while $\beta$ rotates the first $f(n)$ rows of the restriction matrix; this formulation makes it clear that $\phi$ must map $S_{M_n}$ to itself. To prove this formally, use the definition 
\begin{equation*}
S_{M_n} = \left\{ \tau \in S_n \left| \right. \tau(i) \geq g(n) + 1 \text{ for all } i \geq f(n)+1\right\}.
\end{equation*}
Thus, it must be shown that for $i \geq f(n) +1$,
\begin{equation*}
\alpha \tau \beta (i) \geq g(n) +1. 
\end{equation*}
But since $\beta$ only acts on $\{1, 2, \dots, f(n)\}$, for $i \geq f(n)+1$, $\beta(i) = i$. Thus, $\alpha\tau\beta(i) = \alpha \tau(i)$. Since $\tau \in S_{M_n}$, and $i \geq f(n)+1$, $\tau(i) \geq g(n)+1$. Since $\alpha$ only acts on $\{g(n)+1, g(n)+2, \dots, n\}$,  $\alpha \tau(i) = \alpha(\tau(i)) \geq  g(n)+1$. Thus, 
\begin{equation*}
\alpha \tau \beta(i) = \alpha \tau (i) \geq g(n)+1
\end{equation*}
as desired. 
\\
\\
\textbf{Step 2.} It now must be shown that any $\phi$ defined by $\eqref{phidef}$ is an isomorphism. By Step 1, $\phi$ maps into $S_{M_n}$, and it's obviously invertible, so it suffices to show that it preserves edges. That is, for any transposition $(i, j)$ and $\tau \in S_{M_n}$, such that $(i, j) \tau \in S_{M_n}$, $\phi\left((i, j) \tau\right)$ is a neighbor of $\phi(\tau)$. But 
\begin{align*}
\phi((i, j) \tau) &= \alpha (i, j) \tau \beta = (\alpha(i), \alpha(j)) \alpha \tau \beta\\
&= (\alpha(i), \alpha(j))\phi(\tau)
\end{align*}
which is clearly a transposition away from $\phi(\tau)$. Thus, the map $\phi$ preserves edges, and hence is an isomorphism of $S_{M_n}$. 
\\
\\
\textbf{Step 3.} Finally, find $\alpha$ and $\beta$ such that 
\begin{equation*}
\phi(\pi)= \sigma
\end{equation*}
This entails $\alpha \pi \beta = \sigma$. Since $\beta$ acts only on $\{1, 2, \dots, f(n)\}$, this means that for $i > f(n)$,
\begin{align*}
\alpha \pi \beta(i) &= \sigma(i) \\
\Rightarrow \alpha \pi (i) &= \sigma (i) 
\end{align*}
Thus, $\alpha$ must satisfy
\begin{equation} \label{alphadef}
\alpha \pi (i) = \sigma(i) \text{ for } i > f(n)
\end{equation}

This defines $\alpha$ on $\pi(S)$, where $S = \{i > f(n)\}$. Since $\pi \in S_{M_n}$, for all $i > f(n)$, $\pi(i) > g(n)$. Thus, $\pi(S) \subset \{g(n)+1, g(n)+2, \dots, n\}$. Furthermore, for $i > f(n), \sigma(i) > g(n)$: therefore, Equation \eqref{alphadef} says that $\alpha$ must send a subset of $\{g(n)+1, g(n)+2, \dots, n\}$ to some other subset of $\{g(n)+1, g(n)+2, \dots, n\}$. Therefore, it is possible to pick a permutation that acts only on $\{g(n)+1, g(n)+2, \dots, n\}$ and satisfies Equation \eqref{alphadef}: call this permutation $\alpha_0$. 

Now, define $\beta_0 = \pi^{-1} \alpha_0^{-1} \sigma$. Clearly, with this definition, 
\begin{equation*}
\phi(\pi)  = \alpha_0 \pi \beta_0 = \sigma
\end{equation*}
To check that this $\beta_0$ acts only on $\{1, 2, \dots, f(n)\}$, it suffices to show this for $\beta_0^{-1}$.  Let $i > f(n)$: 
\begin{align*}
\beta_0^{-1}(i) &= \sigma^{-1} \alpha_0 \pi(i) = \sigma^{-1}(\sigma(i)) = i
\end{align*}
where the second equality follows by Equation \eqref{alphadef}, which was used to define $\alpha_0$. Thus, $\beta_0^{-1}$, and hence $\beta_0$, fixes $\{f(n)+1, f(n)+2, \dots, n\}$, so $\beta_0$ acts only on $\{1, 2, \dots, f(n)\}$. Therefore, for this choice of $\alpha_0$ and $\beta_0$, $\phi(\tau) = \alpha_0 \tau \beta_0$ is an isomorphism of $S_{M_n}$ which maps $\pi$ to $\sigma$, as required.  
\end{proof}
\begin{rmk}\label{examplenonvertextransitive} 
Unfortunately, $S_M$ is not vertex transitive for all one-sided interval restriction matrices. While I have not discovered an easy characterization for vertex transitivity, it is easy to provide counterexamples. For example, let 
\begin{equation*}
M = \begin{bmatrix} 1 & 1 & 1 & 1 & 1 \\
                    1 & 1 & 1 & 1 & 1 \\
                    1 & 1 & 1 & 1 & 1 \\
                    0 & 1 & 1 & 1 & 1 \\
                    0 & 0 & 1 & 1 & 1 \end{bmatrix}            
\end{equation*}
Both $\sigma = 12345$ and $\tau = 45123$ are in $S_M$. Furthermore, a simple computer calculation shows that 
\begin{equation*}
P^6(\sigma, \sigma) = \frac{5207}{117649} \neq  
P^6(\tau, \tau) = \frac{5287}{117649}
\end{equation*}
implying that the walk is not vertex transitive. I conjecture that $S_M$ is not vertex transitive for almost all one-sided restriction matrices that are not `two-step.'
\end{rmk}

Return to proving Theorem \ref{TVfastmixingexample}, following the outline at the beginning of this section. 

\begin{lemma}\label{Tconditioning}
Let $T$ be the first time the first column is used in the walk. Furthermore, fix $\epsilon$ and assume that $i$ is chosen such that at time $i-1$, the random tranposition walk on $S_{n-1}$ is within $\epsilon$ in total variation distance from uniformity. Then, if $\pi$ is the uniform distribution on $S_{M_n}$ for $t \geq i$,
\begin{equation*}
\frac{1}{2} \sum_{x \in S_{M_n}} \left| \mathbb{P}\left( X_t = x \left| \right. T = i\right) - \pi(x) \right|  < \frac{4}{2f(n)-1}+\epsilon.
\end{equation*}
Thus, the conditional distribution of $X_t$ given $T = i$ is very close in total variation to the uniform distribution. 
\end{lemma}
\begin{proof}[\bf Proof:]
From Lemma \ref{vertextrans}, the walk is vertex transitive. Without loss of generality, the walk starts at the identity. Now, at times $t\geq i$, the walk simply evolves as usual, since the event conditioned on happened already. Since the total variation distance of a Markov chain to stationarity is non-decreasing, for $t \geq i$,  
\begin{align*}
\frac{1}{2} \sum_{x \in S_{M_n}} \left| \mathbb{P}\left( X_t = x \left| \right. T = i\right) - \pi(x) \right| &\leq \frac{1}{2} \sum_{x \in S_{M_n}} \left| \mathbb{P}\left( X_i = x \left| \right. T = i\right) - \pi(x) \right|
\end{align*} 
Clearly, the expression on the right-hand side above is the just total variation distance between $X_i$ conditioned on $T = i$ and stationarity. Thus, it is equal to 
\begin{equation*}\label{TVdistanceconditioned}
\sup_{B \subseteq S_{M_n}} \left| \mathbb{P} \left( X_i \in B \left| \right. T = i\right) - \pi(B) \right|
\end{equation*}
and that is precisely what will be bounded. Accordingly, fix a subset $B$ of $S_{M_n}$. 

By definition, until time $i$ the walk is precisely the random transposition walk on the permutations of $\{2, 3, \cdots, n\}$, and hence it is identical to the random transposition walk on $S_{n-1}$. Recall that $i$ is chosen such that this walk is within $\epsilon$ of stationarity at time $i-1$. Now, let 
\begin{equation}\label{Sdef}
S = \{ \sigma \in S_{M_n} \left| \right. \sigma(1) = 1 \} \cong S_{n-1}
\end{equation}
and let $\tilde{\pi}$ is the uniform distribution on $S$. For any $B \subseteq S$, 
\begin{equation*}
\left| \mathbb{P}(X_{i-1} \in B\left. \right| T = i) - \tilde{\pi}(B) \right| < \epsilon.
\end{equation*}

Now consider the transposition $r_i$ which takes $X_{i-1}$ to $X_i$. Conditioning on $T= i$, this transposition is uniformly distributed between all the transpositions involving the first row (as the walk starts at the identity, this is equivalent to using the first column.) Note that among these, the transposition $(1,1)$ appears once, while the remaining transpositions $(1, k)$ appear twice. This means that there are precisely $2f(n)-1$ possible choices for $r_i$. Now, $r_i$ determines the value of $X_i$ on column $1$. Accordingly, write $B= B_1 \cup B_2 \cup \cdots \cup B_{f(n)}$, where 
\begin{equation*}
B_k = \{\sigma \in B \left. \right| \sigma(k) = 1\}
\end{equation*}
This decomposes $B$ into equivalence classes that depend on the value of $\sigma \in B$ on column $1$. Since $X_{i-1}(1) = 1$, if $r_i = (1,j)$, then $X_i(j) = 1$. Therefore, $X_i$ is in $B_k$ if and only if $r_i$ is $(1,k)$ -- that is, if at time $i$ rows $1$ and $k$ are transposed. Hence, 
\begin{align*}
\mathbb{P}(X_i \in B_k \left. \right| T=i) 
&= \mathbb{P}(X_{i-1}(1, k) \in B_k\left. \right| T=i) \mathbb{P}(r_i = (1,k))\\
&= \mathbb{P}(X_{i-1} \in  B_k(1,k)\left. \right| T=i ) \mathbb{P}(r_i = (1,k))
\end{align*}
As noted earlier, there are precisely $2f(n) - 1$ possibilities for $r_i$, and the cases $k=1$ and $k \neq 1$ are slightly different.  Consider those separately. For $k=1$, 
\begin{align}\label{kisone}
\mathbb{P}(X_i \in B_1\left. \right| T=i) &= \mathbb{P}(X_{i -1} \in B_1 \left. \right| T=i ) \mathbb{P}(r_i = (1,1))\nonumber \\
&= \frac{\mathbb{P}(X_{i -1} \in B_1 \left. \right| T=i)}{2f(n) - 1} \leq \frac{1}{2f(n)-1}
\end{align}
and for $k \neq 1$,  
\begin{align}\label{kisn'tone}
\mathbb{P}(X_i \in B_k\left. \right| T=i ) &=\mathbb{P}(X_{i -1} \in B_k(1,k) \left. \right| T=i ) \mathbb{P}(r_i = (1,k)) \nonumber \\
&= \frac{2 \mathbb{P}(X_{i -1} \in B_k(1,k) \left. \right| T=i)}{2f(n) - 1} 
\end{align}
By choice of $i$, $\mathbb{P}(X_{i-1} \in B_k(1,k) \left. \right| T = i)$ is well-approximated by $\tilde{\pi}(B_k(1,k))$, where $\tilde{\pi}$ is the uniform distribution on $S$ as defined above in Equation \eqref{Sdef}. It is easy to see that $\tilde{\pi}(B_k(1,k)) = f(n) \pi(B_k(1,k)) = f(n)\pi(B_k)$, and thus that 
\begin{equation*}
\left| \mathbb{P}(X_{i-1} \in B_k(1,k) \left. \right| T = i) - f(n) \pi(B_k) \right| < \epsilon
\end{equation*}
for each $k$. 
Therefore, for $k \neq 1$, using Equation \eqref{kisn'tone}, 
\begin{align*}
| \mathbb{P}(X_i \in  B_k\left. \right| T = i) & - \pi(B_k) | = \left| \frac{2 \mathbb{P}(X_{i -1} \in B_k(1,k)\left. \right| T = i) }{2f(n) - 1}  - \pi(B_k) \right|\\
&\leq \frac{2\left| \mathbb{P}(X_{i -1} \in B_k(1,k)\left. \right| T = i)  - f(n) \pi(B_k) \right| + \pi(B_k)}{2f(n)-1}\\
&\leq \frac{2\epsilon +\pi(B_k)}{2f(n)-1}
\end{align*}
Combining the above with Equation \eqref{kisone}, 
\begin{align}\label{TVBbound}
\left| \mathbb{P}(X_{i} \in B \left. \right| T = i) - \pi(B)\right| &= \left| \sum_{k=1}^{f(n)} \left( \mathbb{P}(X_i \in B_k\left. \right| T = i) - \pi(B_k) \right) \right| \nonumber \\
&\leq  \mathbb{P}(X_i \in B_1\left. \right| T = i)  + \pi(B_1) +
\sum_{k=2}^{f(n)} \left( \frac{2\epsilon+ \pi(B_k)}{2f(n)-1} \right) \nonumber \\
&\leq \pi(B_1) +\epsilon +\frac{1}{2f(n)-1}\left(1 +\sum_{k=2}^{f(n)} \pi(B_k)\right)
\end{align}
Now, since $B_1 \subseteq  S$, and $\pi(S)$ is clearly $\frac{1}{f(n)}$,  $\pi(B_1) <  \frac{2}{2f(n)-1}$. Also, 
\begin{equation*}
\sum_{k=2}^{f(n)} \pi(B_k) \leq \pi(B) \leq 1
\end{equation*}
Plugging these back into Equation \eqref{TVBbound},   
\begin{equation*}
\left| \mathbb{P}(X_{i} \in B \left. \right| T = i) - \pi(B)\right| \leq \frac{4}{2f(n)-1}+\epsilon
\end{equation*}
and hence 
\begin{equation*}
\sup_{B \subseteq S_{M_n}} \left| \mathbb{P} \left( X_i \in B \left| \right. T = i\right) - \pi(B) \right| < \frac{4}{2f(n)-1}+\epsilon
\end{equation*}
as required. 
\end{proof} 

Turn now to the proof of Theorem \ref{TVfastmixingexample}. As above, condition on the first time the first column is transposed. 

\begin{proof}[\bf Proof of Theorem \ref{TVfastmixingexample}]  
Without loss of generality, assume that the walk starts at the identity. Define 
\begin{equation}\label{timet}
t = \frac{3(n+2\Delta)}{2f(n)-1}.
\end{equation}
It must be shown that the walk is sufficiently mixed by time $t$. 

Let $T$ is the first time that column $1$ is used in a transposition. Then, 
\begin{align}\label{decompositionbyT}
\left\| P^t(id, \cdot) - \pi \right\|_{TV} &= \frac{1}{2} \sum_{x} \left| P^t(id, x) - \pi(x) \right| \nonumber \\
&= \frac{1}{2}\sum_{x} \left| \sum_{i = 1}^\infty \mathbb{P}(X_t = x \left. \right| T = i)\mathbb{P}(T=i) - \pi(x) \right| \nonumber \\
&= \frac{1}{2}\sum_{x} \left| \sum_{i = 1}^\infty \mathbb{P}(X_t = x \left. \right| T = i)\mathbb{P}(T=i) - \sum_{i=1}^\infty \pi(x)\mathbb{P}(T= i) \right|\nonumber \\
&\leq \sum_{i=1}^\infty \mathbb{P}(T=i) \frac{1}{2}\sum_{x} \left|\mathbb{P}(X_t = x \left. \right| T = i) - \pi(x) \right| 
\end{align}

Now, from the results of Diaconis and Shahshahani, the random transposition walk on $S_{n-1}$ has cutoff at time $\frac{1}{2}(n-1) \log (n-1)$ with a window of size $n-1$. Fix $\epsilon > 0$. If $i\geq n \log n$, then for sufficiently large $n$, the total variation distance between the random transposition walk on $S_{n-1}$ at time $i$ and the uniform distribution is less than $\epsilon$. Thus, by Lemma \ref{Tconditioning} above, if $i \geq n\log n$ and $n$ is sufficiently large, then for $ t\geq i$,
\begin{equation}\label{previouslemma}
\frac{1}{2} \sum_{x \in S_{M_n}} \left| \mathbb{P}\left( X_t = x \left| \right. T = i\right) - \pi(x) \right|  < \frac{4}{2f(n)-1}+\epsilon
\end{equation}
Now, using Equation \eqref{decompositionbyT}, since the total variation distance between distributions is always at most $1$, $\left\|P^t(id, \cdot) - \pi \right\|$ is bounded above by
\begin{align}\label{anotherdecomposition}
 \sum_{i=n \log n }^t \mathbb{P}(T=i) \frac{1}{2}\sum_{x} \left|\mathbb{P}(X_t = x \left. \right| T = i) - \pi(x) \right| + \mathbb{P} ( T \notin [n\log n, t]) 
\end{align}
The probability of transposing the first column in one step is precisely $\frac{2f(n) - 1}{n + 2\Delta}$. Thus, 
\begin{align*}
\mathbb{P}(T \leq n \log n) &= 1 - \mathbb{P}(T > n \log n) \\
                            &= 1 - \left(1 - \frac{2f(n) - 1}{n + 2\Delta}\right)^{n \log n}
\end{align*}
Since $f(n) \leq \frac{n}{5\log n}$ for sufficiently large $n$, and since $n + 2\Delta \approx n^2$,  
\begin{equation*}
1 - \left(1 - \frac{2f(n) - 1}{n + 2\Delta}\right)^{n \log n} \leq 1 - e^{-1/2}
\end{equation*}
for sufficiently large $n$. Similarly, the probability that $T > t$ is just 
\begin{equation*}
\left(1 - \frac{2f(n) - 1}{n + 2\Delta}\right)^{t} \leq \exp \left(- \frac{t(2f(n) - 1)}{n+2\Delta}\right) = e^{-3}
\end{equation*} 
and thus $\mathbb{P}\left( T \notin [n\log n, t]\right) \leq 1 - e^{-1/2} + e^{-3} < 0.45$. Next, by Equations \eqref{anotherdecomposition} and \eqref{previouslemma}, for sufficiently large $n$,
\begin{align*}
\left\| P^t(id, \cdot) - \pi \right\|_{TV} &\leq \sum_{i=n \log n}^{t} \mathbb{P}(T=i) \frac{1}{2}\sum_{x} \left|\mathbb{P}(\alpha_t = x \left. \right| T = i) - \pi(x) \right| + 0.45 \\
&\leq \sum_{i=n \log n}^{t} \mathbb{P}(T=i) \left(  \frac{4}{2f(n)-1}+\epsilon\right) + 0.45 \\
&\leq \frac{4}{2f(n)-1}+\epsilon + 0.45 
\end{align*}
Since $\epsilon$ can be chosen to be anything, and $f(n) \rightarrow \infty$, the above is clearly less than $\frac{1}{2}$ for sufficiently large $n$. Thus, the walk has mixed by time $t=  \frac{3(n+2\Delta)}{2f(n)-1}$, which is clearly of order $\frac{n+2\Delta}{f(n)}$. 

Finally, the walk does not have cutoff: starting  at the identity,   
\begin{equation*}
\mathbb{P} \left( X_k(1) = 1 \right) \geq \mathbb{P}(T > k) = \left(1 - \frac{2f(n) - 1}{n + 2\Delta}\right)^k
\end{equation*}
Under the uniform distribution $\pi$, the probability of $S = \{\sigma \in S_{M_n} \left| \right. \sigma(1) = 1\}$ is precisely $\frac{1}{f(n)}$. This means that 
\begin{equation*}
\left\| P^k(id, \cdot) - \pi\right\|_{TV} \geq \left(1 - \frac{2f(n) - 1}{n + 2\Delta}\right)^k - \frac{1}{f(n)}
\end{equation*}
It is easy to see that the above function falls off smoothly as opposed to exhibiting cutoff. For example, at time $2t = \frac{6(n+2\Delta)}{2f(n) - 1}$, it is approximately $e^{-6}$, which does not approach $0$ as $n \rightarrow \infty$. Thus, the walk does not have cutoff, completing the proof. 
\end{proof}

\section{Hanlon's Results and Other Preliminaries} \label{hanlonpreviouswork}

The remainder of this paper focuses on proving Theorem \ref{chitheorem}, which is concerned with chi-squared cutoff. This section contains a review of Hanlon's work and other background material, which shows how to diagonalize the adjacency matrix for $S_{M(\vec{b})}$. The following is Definition 4.2 from his paper \cite{HanlonPaper}: 
\begin{defn}
Let $\vec{b} = (b_1, \dots, b_n)$ be a sequence chosen from $\{1, 2, \dots, n\}$ satisfying $b_1 \leq b_2 \leq \dots \leq b_n$. Call $u$ and $v$ left-equivalent if $b_u = b_v$. Let $L_1, L_2, \dots, L_s$ denote the left-equivalence classes of $\{1, 2, \dots, n\}$, where the ordering is chosen so that whenever $i < j$, the elements of $L_i$ are less than the elements of $L_j$. 

For notational simplicity, define $b_{n+1} = n+1$, and call $u$ and $v$ right-equivalent if there exist $i$ and $i+1$ such that $b_i \leq u, v \leq b_{i+1}$. It is straightforward to check that the number of left-equivalence classes is equal to the number of right-equivalence classes. Let $R_1, \dots, R_s$ denote the right-equivalence classes of $\{1, 2, \dots, n\}$, where the $R_i$ are ordered in the same way as the $L_i$ above. 
\end{defn}

\begin{example}\label{Mofbexample}
Left-equivalence and right-equivalence classes are very easy to visualize via the restriction matrix $M(\vec{b})$. For example, let $\vec{b} = (1, 1, 1, 2, 4)$. Then, the corresponding $M(\vec{b})$ is below:
\begin{equation*}
M(\vec{b}) = \begin{bmatrix} 1 & 1 & 1 & 1 & 1 \\
                    1 & 1 & 1 & 1 & 1 \\
                    1 & 1 & 1 & 1 & 1 \\
                    0 & 1 & 1 & 1 & 1 \\
                    0 & 0 & 0 & 1 & 1 \end{bmatrix}
\end{equation*}

Now, imagine the restriction matrix as an $n \times n$ chessboard, where each square contains either a $0$ or a $1$. Shade in the squares that contain $1$s, and look at the southwest boundary of the shaded area. This has a `step pattern': in the above example, go down 3 steps, right 1 step, down 1 step, right 2 steps, down 1 step, right 2 steps. It is easy to check that the left-equivalence classes correspond to the down stretches of the step pattern, and the right-equivalence classes correspond to the stretches pointing right. 

Using this, for $\vec{b}$ as defined above, the left-equivalence classes are $\{1, 2, 3\}, \{4\}$ and $\{5\}$, while the right-equivalence classes are $\{1\}, \{2,3\}$ and $\{4,5\}$. (This visualization also explains why the numbers of left-equivalence and right-equivalence classes match.)
\end{example}

\begin{rmk}\label{leftandrightclasses}
It should be clear that if $\vec{b}_n(f, g)$ is defined as in Definition \ref{boffg}, then $M_n$ is a `two-step' restriction matrix: that is, the southwest boundary described above will go down for $f(n)$ steps, then will go right for $g(n)$ steps, then will go down for $n - f(n)$ steps, then will go right for $n - g(n)$ steps. Thus, 
\begin{equation}\label{leftoffg}
L_1 = \{1, 2, \dots, f(n)\}, L_2  = \{f(n)+1, \dots, n\}
\end{equation}
and 
\begin{equation}\label{rightoffg}
R_1 = \{1, 2, \dots, g(n)\}, L_2  = \{g(n)+1, \dots, n\}
\end{equation}
For example, if $n = 5$, $f(5) = 3$ and $g(5) = 2$, then $\vec{b} = (1, 1, 1, 3, 3)$ as in Equation \eqref{boffgexample} above. From the $M(\vec{b})$ below it, it's clear that the left-equivalence classes are $\{1, 2, 3\}$ and $\{4,5\}$, while the right-equivalence classes are $\{1, 2\}$ and $\{3,4,5\}$ matching the expressions above. 
\end{rmk}

The following is Definition 4.3 from Hanlon \cite{HanlonPaper}. 

\begin{defn}\label{defbpartition}
A $\vec{b}$-partition $\alpha = (\lambda_1, \mu_1, \lambda_2, \dots, \mu_{s-1}, \lambda_s)$ is a sequence of partitions such that
\begin{enumerate}
\item $\lambda_i \supseteq \mu_i \subseteq \lambda_{i+1}  $ for all $i = 1, 2, \dots, s-1$
\item $\left| \lambda_i\setminus \mu_i \right|= \left| R_i \right|$ for all $i = 1, 2, \dots, s$ (defining $\mu_s = \emptyset$)
\item $\left| \lambda_{i+1} \setminus \mu_i \right| = \left| L_{i+1} \right|$ for all $i = 0, 1, \dots, s-1$ (defining $\mu_0 = \emptyset$)
\end{enumerate}
\end{defn}

It is helpful to have some standard terminology (a good reference for this is Stanley \cite{StanleyVol1}). Recall that a \textit{Ferrers diagram} is a finite collection of boxes, arranged in left-justified rows, such that each row has at least as many boxes as the row directly below it. If $(a_1, a_2, \dots, a_m)$ is a partition, then the Ferrers diagram associated to it has precisely $a_i$ boxes in row $i$. For example, below is the Ferrers diagram of the partition $(4,2,1)$ of $7$:

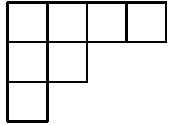
\begin{picture}(250,60)(20,-4)
`
\linethickness{0.6 pt}
\put(150, 0){\line(0,1){45}}
\put(150, 45){\line(1, 0){60}}
\put(165, 0){\line(0,1){45}}
\put(180, 15){\line(0,1){30}}
\put(195, 30){\line(0,1){15}}
\put(210, 30){\line(0,1){15}}
\put(150, 30){\line(1, 0){60}}
\put(150, 15){\line(1,0){30}}
\put(150, 0){\line(1, 0){15}}
\end{picture}

The \textit{transpose} $\lambda^{T}$ of a partition $\lambda$ is a partition whose Ferrers diagram is a reflection of the Ferrers diagram of $\lambda$ along the main diagonal $y = -x$. For example, for $\lambda = (4,2,1)$, the tranpose partition $\lambda^T$ is $(3,2,1,1)$.

Ferrers diagrams give a good way of visualizing $\vec{b}$-partitions. If the left-equivalence classes are the sets $L_1, L_2, \dots, L_s$ and the right-equivalence classes are the sets $R_1, R_2, \dots, R_s$, then to get a $\vec{b}$-partition start with $\mu_0 = \emptyset$, add $|L_1|$ squares to get to $\lambda_1$, delete $|R_1|$ squares to get $\mu_1$, add $|L_2|$ squares to get $\lambda_2$, etc. For example, if $\vec{b} = (1, 1, 1, 3, 3)$, with left-equivalence classes $\{1, 2, 3\}$ and $\{4,5\}$, and right-equivalence classes $\{1, 2\}$ and $\{3,4,5\}$, then the following is a $\vec{b}$-partition: 

\begin{picture}(200,50)(-30,-5)
\put(-30, 20){$\mu_0 = \emptyset$,}

\put(20, 20){$\lambda_1 = $} 
\linethickness{0.6 pt}
\put(45, 15){\line(1,0){45}}
\put(45, 30){\line(1,0){45}}

\put(45, 15){\line(0, 1){15}}
\put(60, 15){\line(0,1){15}}
\put(75, 15){\line(0,1){15}}
\put(90, 15){\line(0,1){15}}
\put(91, 15){,}

\put(115, 20){$\mu_1 = $} 
\linethickness{0.6 pt}
\put(140, 15){\line(1,0){15}}
\put(140, 30){\line(1,0){15}}

\put(140, 15){\line(0, 1){15}}
\put(155, 15){\line(0,1){15}}
\put(156, 15){,}

\put(180, 20){$\lambda_2 = $} 
\linethickness{0.6 pt}
\put(205, 0){\line(1,0){15}}
\put(205, 15){\line(1,0){30}}
\put(205, 30){\line(1,0){30}}

\put(205, 0){\line(0, 1){30}}
\put(220, 0){\line(0,1){30}}
\put(235, 15){\line(0,1){15}}
\put(236, 15){,}

\put(260, 20){$\mu_2 = \emptyset$}

\end{picture}

A labeling of a Ferrers diagram is a Ferrers diagram with numbers filled into the boxes. A \textit{standard Young tableau} is a labeling in which the entries of each row are strictly increasing from left to right, while the entries of each column are strictly increasing from top to bottom. A \textit{semi-standard Young tableau} has strictly increasing columns, but weakly increasing (nondecreasing) rows. The following is a semi-standard Young tableau for the partition $(4, 2, 1)$ of $7$: 

\begin{picture}(250,60)(20,-4)
`
\linethickness{0.6 pt}

\put(150, 0){\line(0,1){45}}
\put(150, 45){\line(1, 0){60}}
\put(165, 0){\line(0,1){45}}
\put(180, 15){\line(0,1){30}}
\put(195, 30){\line(0,1){15}}
\put(210, 30){\line(0,1){15}}
\put(150, 30){\line(1, 0){60}}
\put(150, 15){\line(1,0){30}}
\put(150, 0){\line(1, 0){15}}
\put(150, 0){\line(0,1){45}}

\put(155, 35){1}
\put(170, 35){1}
\put(185, 35){2}
\put(200, 35){3}
\put(155, 20){2 }
\put(170, 20){4}
\put(155, 5){3}

\end{picture}

\begin{defn}
Let $\alpha$ and $\beta$ be partitions satisfying $\alpha_i \geq \beta_i$ for all $i$. Let $n = |\alpha| - |\beta|$. Then, the skew-shape $\alpha/\beta$ is well-defined, and there exists a representation of $S_n$ corresponding to $ \alpha/\beta$. Denote the degree of this representation by $\left| X_{\alpha/\beta}\right|$. 
\end{defn}

\begin{rmk}\label{degskewrep}
It is well-known that $\left|X_{ \alpha/\beta}\right|$ is the number of standard Young tableaux of shape  $\alpha/\beta$ (a good source for this and results like it is Macdonald's book \cite{MacdonaldBook}). 
\end{rmk}

\begin{defn}\label{indicatortableau}
Let $\alpha = (\lambda_1, \mu_1, \lambda_2, \dots, \lambda_s)$ be a $\vec{b}$-partition. Define the indicator tableau $T(\alpha)$ to be the tableau whose entry in a square $x$ is the number of skew shapes $\lambda_i/\mu_{i-1}$ that contain $x$. Denote this entry by $T_x(\alpha)$. 
\end{defn}
Now, recall that the \textit{content} of a square $x$ in a Ferrers diagram is denoted by $c_x$, and is defined to be $j - i$, where $x$ lies in column $j$ from the left, and row $i$ from the top. The following theorem restates a main result of Hanlon's paper, Theorem 4.15, which derives all the eigenvalues of the adjacency matrix $U$ using the language indicated above. 

\begin{theorem}\label{hanlonthm}
For every $\vec{b}$-partition $\alpha = (\lambda_1, \mu_1, \lambda_2, \dots, \lambda_s) $, there exists an eigenspace $R_n(\alpha)$ of $U(\vec{b})$ such that the vector space $\mathbbm{C} S_{M(\vec{b})}$ decomposes as a direct sum of the spaces $R_n(\alpha)$; that is,  
\begin{equation*}
\mathbbm{C} S_{M(\vec{b})} = \bigoplus_{\alpha} R_n(\alpha)
\end{equation*} where the sum runs over all $\vec{b}$-partitions $\alpha$. The dimension of $R_n(\alpha)$ is equal to
\begin{equation}\label{dimeig}
\prod_{i=0}^s \left| X_{\lambda_i/\mu_i} \right| \left| X_{\lambda_{i+1}/\mu_{i}}\right|
\end{equation}
where $\mu_0 = \mu_s = \emptyset$. 

Furthermore, letting the eigenvalue of $U$ corresponding to $R_n(\alpha)$ be $\Lambda(\alpha)$, 
\begin{equation}\label{eigvalue}
\Lambda(\alpha)= \sum_{x} T_x(\alpha)c_x
\end{equation}
where the sum is over all the squares $x$ in the indicator tableau $T(\alpha)$.
\end{theorem}

The following lemma finds another expression for the eigenvalue $\Lambda(\alpha)$. It requires the following definition:
\begin{defn}\label{contentdef}
If $\lambda$ is a partition, define $C(\lambda)$ to be the sum of the contents of all the squares of $\lambda$. 
\end{defn}       

\begin{lemma}\label{eigformula}
If $\alpha = (\lambda_1, \mu_1, \dots, \mu_{s-1}, \lambda_s)$ is a $\vec{b}$-partition, 
\begin{align}\label{bettereigformula} 
\Lambda(\alpha) &= C(\lambda_1) - C(\mu_1) + \dots - C(\mu_{s-1}) +C(\lambda_2) \nonumber \\
        &= \sum_{i =1}^s C(\lambda_i) - \sum_{i=1}^{s-1} C(\mu_i)
\end{align}
\end{lemma}
\begin{proof}[\bf Proof:]
From Equation \eqref{eigvalue}, 
\begin{equation*}
\Lambda(\alpha)= \sum_{x} T_x(\alpha)c_x
\end{equation*}
where the sum is over all the squares $x$ in the indicator tableau $T(\alpha)$. Recall that $T_x(\alpha)$ is the number of skew shapes $\lambda_i/\mu_{i-1}$ that contain $x$. Thus, since $C(\mu_0) = 0$, 
\begin{align*}
\sum_{i =1}^s C(\lambda_i) - \sum_{i=1}^{s-1} C(\mu_i) &= \sum_{i = 1}^s \sum_{x \in \lambda_i}c_x - \sum_{i=0}^{s-1}\sum_{x\in \mu_i} c_x \\
 &= \sum_{i = 1}^s \left(\sum_{x \in \lambda_i}c_x - \sum_{x \in \mu_{i-1}} c_x\right) \\
&= \sum_{i=1}^s \sum_{x \in \lambda_i/\mu_{i-1}} c_x = \sum_{i=1}^s \sum_{x \in T(\beta)} c_x \mathbf{1}_{\{x\in \lambda_i/\mu_{i-1}\}} \\
&= \sum_{x\in T(\beta)} c_x \sum_{i=1}^s  \mathbf{1}_{\{x\in \lambda_i/\mu_{i-1}\}} = \sum_{x \in T(\beta)} c_x T_x(\alpha) \\
&= \Lambda(\alpha)
\end{align*}
as required. 
\end{proof}

To diagonalize the transition matrix $P$, make the following simple definition. 

\begin{defn}
Define 
\begin{equation}\label{eigsP}
\Lambda_1(\alpha) = \frac{n + 2\Lambda(\alpha)}{n+2\Delta}
\end{equation} 
From Equation \eqref{Ptransitionmatrix}, 
\begin{equation*}
P = \frac{1}{n+ 2\Delta}(nI + 2 U),
\end{equation*}
so $\Lambda_1(\alpha)$ are the eigenvalues of $P$. 
\end{defn}

\section{Bounding the Eigenvalues}\label{boundingtheeigenvalues}

Lemma \ref{vertextrans} above shows that the graph on $S_{M(\vec{b})}$ induced by the walk is vertex transitive. Thus, Corollary \ref{alleigsvertextransitive} and Theorem \ref{hanlonthm} combine to show that 
\begin{equation}\label{alleigsequation}
 \left\| P^t(x, \cdot) - \pi \right\|_{2, \pi} = \sqrt{\sum_{\alpha} \dim(R_n(\alpha)) \Lambda_1(\alpha)^{2t}}
\end{equation}
where the sum is over $\vec{b}$-partitions $\alpha$ such that $\Lambda_1(\alpha) \neq 1$. Using the above expression calls for good bounds on the eigenvalues $\Lambda_1(\alpha)$ and the dimensions of the eigennspaces $R_n(\alpha)$. This section concentrates on the eigenvalues. It is first shown that it suffices to consider non-negative eigenvalues $\Lambda_1(\alpha)$ by showing that the eigenvalues come in pairs. Since this is always true, it is proved for a general $\vec{b}$: 

\begin{lemma}
Let $\alpha = (\lambda_1, \mu_1, \lambda_2, \dots, \lambda_s)$ be a $\vec{b}$-partition. Letting
\begin{equation*}
\alpha^T = (\lambda_1^T, \mu_1^T, \lambda_2^T, \dots, \lambda_s^T)
\end{equation*} 
it can be concluded that
\begin{equation}
\dim(R_n(\alpha^T)) = \dim(R_n(\alpha)) \text{ and } \Lambda(\alpha^T) = - \Lambda(\alpha)
\end{equation}
\end{lemma}
\begin{proof}[\bf Proof:]
From Theorem \ref{hanlonthm}, 
\begin{equation*}
\dim(R_n(\alpha^T)) = \prod_{i=0}^s \left| X_{\lambda_i^T/\mu_i^T} \right| \left| X_{\lambda_{i+1}^T/\mu_{i}^T}\right|
\end{equation*}
It is well-known that for general $\alpha \supseteq \beta$, $X_{\alpha^T/\beta^T}$ and  $X_{\alpha/\beta}$ are conjugate representations, and hence their dimensions are equal. Thus,
\begin{equation*}
\dim(R_n(\alpha^T)) = \prod_{i=0}^s \left| X_{\lambda_i/\mu_i} \right| \left| X_{\lambda_{i+1}/\mu_{i}}\right| = \dim(R_n(\alpha))
\end{equation*}
Furthermore, it is clear from the definition that $C(\alpha^T) = -C(\alpha)$. Lemma \ref{bettereigformula} gives
\begin{align*}
\Lambda(\alpha^T) &= \sum_{i =1}^s C(\lambda_i^T) - \sum_{i=1}^{s-1} C(\mu_i^T) = - \sum_{i =1}^s C(\lambda_i) + \sum_{i=1}^{s-1} C(\mu_i) \\
  &= -\Lambda(\alpha)  
\end{align*}
as required. 
\end{proof}

Now, assume that $\Lambda_1(\alpha) < 0$. Then, 
\begin{align*}
\left| \Lambda_1(\alpha) \right| &= -\frac{n + 2\Lambda(\alpha)}{n+2\Delta} = \frac{-n +  2\Lambda(\alpha^T)}{n+2\Delta} 
                \leq \frac{n + 2\Lambda(\alpha^T)}{n + 2 \Delta} \\
                &= \Lambda_1(\alpha^T)
\end{align*}
Therefore, 
\begin{equation}\label{positiveeigs}
\left\| P^t(x, \cdot) - \pi \right\|_{2, \pi} \leq \left( 2\sum_{1 \neq \Lambda_1(\alpha) \geq 0} \dim(R_n(\alpha)) \Lambda_1(\alpha)^{2t}\right)^{1/2}
\end{equation}
Hence, it will suffice to provide upper bounds for the eigenvalues. The argument follows the approach laid out by Diaconis and Shahshahani in their analysis for $S_n$ \cite{DiaconisandShah}. For their walk, the eigenvalues were functions of partitions $\lambda$, and it turned out that a good bound for these eigenvalues could be derived using only the largest part of $\lambda$. A similar trick is used to bound $\Lambda_1(\alpha)$. 

As noted earlier in Remark \ref{leftandrightclasses}, for $\vec{b} = \vec{b}_n(f, g)$, there are only two left-equivalence and two right-equivalence classes. Furthermore, by Equations \eqref{leftoffg} and \eqref{rightoffg}, $|L_1| = f(n), |L_2| = n - f(n)$, and $|R_1| =  g(n), |R_2| = n-g(n)$. Thus, by Definition \ref{defbpartition}, a $\vec{b}$-partition $\alpha$ can be written as 
\begin{align}\label{boffgpartition}
\alpha = (\lambda_1, \mu_1, \lambda_2), \ \
&\text{ where } |\lambda_1| = f(n), |\mu_1| = f(n) - g(n), |\lambda_2| = n -g(n),\\ 
&\text{ and } \lambda_1 \supseteq \mu_1 \subseteq \lambda_2 \nonumber
\end{align} 
Hence, by Lemma \ref{eigformula}, for $\alpha = (\lambda_1, \mu_1, \lambda_2)$,
\begin{equation}\label{lambdaexpression}
\Lambda(\alpha) = C(\lambda_1) - C(\mu_1) + C(\lambda_2)
\end{equation}
Since by Equation \eqref{eigsP}, $\Lambda_1(\alpha)$ is just $(n+2\Lambda(\alpha))/(n+2\Delta)$, it suffices to find bounds on the above quantity. 

The largest parts of the partitions will be used to bound $\Lambda(\alpha)$. For sequences of partitions $\alpha = (\lambda_1, \mu_1, \lambda_2)$, the bounds will depend on the sequence of largest parts of $\lambda_1$, $\mu_1$ and $\lambda_2$. In order to slightly simplify calculations, instead of working directly with the largest parts of $\alpha$, it is convenient to work with what remains when the largest part has been taken away. The following definitions and lemmas carry out these ideas. 

\begin{defn}
Let $\lambda$ be a partition of $n$. Then, write $\lambda$ as
\begin{equation*}
\lambda = (\lambda^1, \lambda^2, \lambda^3, \dots)
\end{equation*}
where $\lambda^1 \geq \lambda^2 \geq \lambda^3 \dots$; in particular, the largest part of $\lambda$ will be denoted by $\lambda^1$. Futhermore, note that $(\lambda^2, \lambda^3, \dots)$ is a partition of $n - \lambda^1$. Call this partition the remainder of $\lambda$, and denote it by $\lambda^{Re}$. 

Now, let $\alpha$ be a $\vec{b}$-partition, where $\alpha = (\lambda_1, \mu_1, \lambda_2)$. Then, define 
\begin{align*}
\alpha^1 &= \left(\lambda_1^1, \mu_1^1, \lambda_2^1\right) \\
\alpha^{Re} &= \left(\lambda_1^{Re}, \mu_1^{Re}, \lambda_2^{Re}\right)
\end{align*}
analogously to above. Furthermore, define
\begin{equation*}
\left| \alpha^{Re} \right| = \left(f(n) - \lambda_1^1, f(n) - g(n) - \mu_1^1, n -g(n) - \lambda_2^1\right)
\end{equation*}
Note that 
\begin{equation*}
\left| \alpha^{Re} \right| = \left( \left| \lambda_1^{Re}\right| , \left|\mu_1^{Re}\right| , \left| \lambda_2^{Re}\right| \right) 
\end{equation*}
\end{defn}

The eigenvalue $\Lambda(\alpha)$ is bounded by finding a function of $\left|\alpha^{Re} \right| $ which is greater than it. The following lemmas carry this out. 
\begin{lemma} \label{skewpartitionbound}
Let $\lambda$ be a partition of $l$ and $\mu$ be a partition of $m$ such that $\lambda \supseteq \mu$. Assume that 
\begin{align*}
\left| \lambda^{Re} \right| = i \mbox{ and }
\left| \mu^{Re} \right| = j
\end{align*}
Then, for $C(\lambda)$ as defined in Defintion \ref{contentdef}, the following inequality holds:
\begin{equation*}
C(\lambda) - C(\mu) \leq \frac{l^2 - l}{2} - \frac{m^2 - m}{2} - i(l - i +1) + j(m - j+1)
\end{equation*}
Furthermore, if $j \leq m/2$ and $i \leq l/2$ then equality is achieved at $\lambda = (l-i, i)$ and $\mu  = (m-j, j)$.
\end{lemma}

\begin{proof}[\bf Proof:]

Let $sq(a, b)$ denote the square of the Ferrers diagram in column $a$ and row $b$, with the square in the upper left corner denoted by $sq(1, 1)$. The content of a square $x = sq(a, b)$ is $a-b$. This increases as the column $a$ increases, and decreases as the row $b$ increases. 

Note that $C(\lambda) - C(\mu) = C(\lambda/\mu)$, since $\lambda \supseteq \mu$. Since $\left| \lambda^{Re} \right| = i$ and
$\left| \mu^{Re}\right| = j$, $\lambda_1 = l - i$ and $\mu_1 = m - j$. Thus, the first row of $\lambda/\mu$ contains the squares $sq(m-j+1,1), sq(m-j+2, 1), \dots, sq(l-i, 1)$. Furthermore, $\lambda/\mu$ contains another $i - j$ squares in other rows. 

Now, start with the Ferrers diagram that consists only of the first row squares $sq(m-j+1,1), sq(m-j+2, 1), \dots, sq(l-i, 1)$ and add squares until reaching a skew partition that contains $l - m$ squares. Since $\left| \mu^{Re} \right|  = j$, the first square added must be at most in column $j+1$ and at least in row $2$. By similar logic, the $r$th square added must be at least in column $j + r$ and at least in row $2$. Thus, the content of the $r$th square we add is at most $j+ r -2$, so 
\begin{align*}
C(\lambda/\mu) & = C(\text{squares in first row}) + C(\text{remaining squares}) \\
           &\leq  \left( (m-j) + \dots + (l - i - 1)\right) + \left( (j-1) + \dots + (i-2) \right) \\
            &= {l - i \choose 2} - {m - j \choose 2} + {i - 1 \choose 2} - {j -1 \choose 2}\\
           &= \frac{l^2 - l}{2} - \frac{m^2 - m}{2} - i(l - i +1) + j(m - j+1)
\end{align*}
as required. Furthermore, if $j \leq m/2$ and $i \leq l/2$, then is straightforward to check that $\lambda = (l-i, i)$ and $\mu = (m-j, j)$ are partitions, satisfy $\lambda \subseteq \mu$, and achieve equality. 
\end{proof}
\begin{example}
Let $l = 9$ and $i = 4$, and let $m = 5$ and $j = 2$. Then, $\lambda$ has $5$ squares in the first row, and $\mu$ has $3$. This $\lambda/\mu$ has $2$ squares in the first row, and $2$ more squares somewhere below the first row. Clearly, the first row of $\lambda/\mu$ looks like 

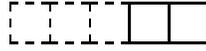
\begin{picture}(300,40)(-75,-4)
`
\linethickness{0.6 pt}
\multiput(50, 5)(15,0){3}{\dashbox{3}(15, 15)}
\multiput(95, 5)(15,0){3}{\line(0,1){15}}
\put(95,5){\line(1,0){30}}
\put(95,20){\line(1,0){30}}
\end{picture}

\noindent where the dotted lines correspond to the squares of $\mu$. Now, the remaining $2$ squares of $\lambda/\mu$ need to placed somewhere. By definition of content, $C(\lambda/\mu)$ is maximized by placing these squares as far up and as far to the right as possible. Clearly the best thing to do is to put all the remaining squares in the second row -- it is easy to see that otherwise, a square can be moved up and to the right, increasing content. Furthermore, since $\mu$ has $2$ squares below the first row, all of these should be put in the second row as well, to make sure that the squares of $\lambda/\mu$ are as far to the right as possible. Thus, the optimal arrangement is 

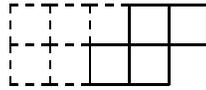
\begin{picture}(20,50)(-75,0)
`
\linethickness{0.6 pt}
\multiput(50, 20)(15,0){3}{\dashbox{3}(15, 15)}
\multiput(95, 20)(15,0){3}{\line(0,1){15}}
\multiput(50, 5)(15,0){2}{\dashbox{3}(15, 15)}
\put(95,35){\line(1,0){30}}
\put(80,20){\line(1,0){45}}
\put(80,5){\line(1,0){30}}
\put(80,5){\line(1,0){30}}
\put(80,5){\line(0,1){15}}
\put(95,5){\line(0,1){15}}
\put(110,5){\line(0,1){15}}
\end{picture}

\noindent which is precisely the case $\lambda = (5,4)$ and $\mu = (3,2)$. Note that this arrangement would have been impossible if $i$ or $j$ was too large, but the same argument would have still provided an upper bound. 

\end{example}

\begin{cor}\label{maxskewpartition} 
Let $\lambda$ be a partition of $l$ and $\mu$ be a partition of $m$ such that $\lambda \supseteq \mu$. Then, 
\begin{equation*}
C(\lambda) - C(\mu) \leq \frac{l^2 - l}{2} - \frac{m^2 - m}{2}
\end{equation*}
Furthermore, equality is achieved at $\lambda = (l)$ and $\mu = (m)$. 
\end{cor}
\begin{proof}[\bf Proof:]
A proof identical to the one for Lemma \ref{skewpartitionbound} above works, but instead the lemma itself is used. Assume that $\left| \lambda^{Re} \right| = i$ and $\left| \mu^{Re} \right| = j$. Since $\lambda \supseteq \mu$, it follows that $i \geq j$ and $\lambda^1 \geq \mu^1$, so $l - i \geq m - j.$ Combining, $i(l-i+1) \geq j(m-j+1)$. Thus, 
\begin{align*}
C(\lambda) - C(\mu) &\leq \frac{l^2 - l}{2} - \frac{m^2 - m}{2} - i(l - i +1) + j(m - j+1)\\
                    &\leq \frac{l^2 - l}{2} - \frac{m^2 - m}{2}
\end{align*}
as required. It is easy to check that $\lambda = (l)$ and $\mu = (m)$ achieve equality.  
\end{proof}

The following result appears in Diaconis (Lemma 2, Chapter 3) \cite{DiaconisBook}. A proof is given.

\begin{lemma}\label{partitionbound}
Let $\lambda$ be a partition of $l$, and let $C(\lambda)$ be the sum of the contents of the squares in $\lambda$, as defined above. Let $\left| \lambda^{Re} \right| = i$. Then, 
\begin{equation*}
C(\lambda) \leq \frac{l^2 - l}{2} - i(l -i +1)                      
\end{equation*} 
for all $i$, with equality achieved for $i \leq \frac{l}{2}$ at $\lambda = (l-i, i)$.

Furthermore, if $i \geq l/2$, 
\begin{equation*}
C(\lambda) \leq \frac{l^2 - l}{2} - \frac{il}{2}
\end{equation*}
\end{lemma}

\begin{proof}[\bf Proof:] 

The first inequality follows easily from Lemma \ref{skewpartitionbound}. Let $\mu = \emptyset$. Then, $m = j = 0$, and  
\begin{equation*}
C(\lambda) \leq \frac{l^2 - l}{2} - i(l -i +1) 
\end{equation*}
as required. From the same lemma, if $i \leq l/2$, equality is achieved at $(l-i,i)$. 

To prove the second inequality, note that if $\lambda = (\lambda^1, \lambda^2, \dots ,\lambda^r)$, then 
\begin{equation*}
C(\lambda) = \sum_{s=1}^r \frac{\lambda^s(\lambda^s - s-1)}{2}
\end{equation*}
Clearly, $\lambda^s \leq \lambda^1$ for each $s$. Thus, 
\begin{align*}
C(\lambda) &\leq \sum_{s=1}^r \frac{\lambda^s(\lambda^1 - s -1)}{2}\leq 
                               \sum_{s=1}^r \frac{\lambda^s(\lambda^1 - 1)}{2}\\
           &= \frac{\lambda^1 - 1}{2} \sum_{s=1}^r \lambda^s 
\end{align*}
Now, $\lambda^1 - 1 = l - i-1$, while $\sum \lambda^s$ is the total number of squares in the partition $\lambda$, which is equal to $l$. This gives 
\begin{equation*}
C(\lambda) \leq \frac{(l - i -1)l}{2} = \frac{l^2 - l}{2} - \frac{il}{2}
\end{equation*}
as required. 
\end{proof}

The next two lemmas bound the eigenvalues $\Lambda_1(\alpha)$. The first lemma will be used when $\left|\lambda_2^{Re}\right|$ is sufficiently large, and the second lemma will be used the rest of the time. 

\begin{lemma}\label{kbig}
 Let $\alpha = (\lambda_1, \mu_1, \lambda_2)$ be a $\vec{b}$-partition, where $\vec{b} = \vec{b}_n(f, g)$, such that $\left|\lambda_2^{Re} \right| = k \geq (n- g(n))/5$. Then, for sufficiently large $n$, 
\begin{equation*}
\Lambda_1(\alpha) \leq \frac{9}{10} 
\end{equation*}
\end{lemma}
\begin{proof}[\bf Proof:]
From Equation \eqref{lambdaexpression}, 
\begin{equation*}
\Lambda(\alpha) = C(\lambda_1) - C(\mu_1) + C(\lambda_2)
\end{equation*}
From Equation \eqref{eigsP},
\begin{equation*}
\Lambda_1(\alpha) = \frac{n + 2\Lambda(\alpha)}{n + 2 \Delta}
\end{equation*}
and from Lemma \ref{deltafg}, 
\begin{equation*}
 n + 2 \Delta = n^2 - 2ng(n) +2g(n)f(n)
\end{equation*}
Since $\lambda_2$ is a partition of $n - g(n)$, and $\left|\lambda_2^{Re} \right| = k \geq (n- g(n))/5$, Lemma \ref{partitionbound} gives
\begin{equation*}
C(\lambda_2) \leq \frac{(n-g(n))^2 - (n-g(n))}{2} -\frac{k(n - g(n))}{2} \leq  \frac{2(n - g(n))^2}{5}
\end{equation*}
Further, from Corollary \ref{maxskewpartition},
\begin{align*}
C(\lambda_1) - C(\mu_1) &\leq \frac{f(n)^2-f(n)}{2} - \frac{(f(n) - g(n))^2 - (f(n) - g(n))}{2}
\end{align*} 
Since $f(n)/n$ approaches $0$, $g(n) \leq f(n)$, and  $ n+ 2\Delta = n^2 - 2ng(n) +2g(n)f(n)$,  the above expression for $C(\lambda_1) - C(\mu_1)$ is negligible compared to $n+2\Delta$. Similarly, $n/(n+2\Delta)$ approaches $0$ as $n\rightarrow \infty$,
\begin{align*}
\lim_{n\rightarrow \infty} \Lambda_1(\alpha) &= \lim_{n \rightarrow \infty} \frac{n+ 2C(\lambda_1) - 2C(\mu_1) + 2C(\lambda_2)}{n+2\Delta} \\
                 &= \lim_{n\rightarrow \infty} \frac{4}{5}\left( \frac{(n-g(n))^2}{n^2 - 2n g(n) + 2f(n)g(n)} \right) \\ 
&= \frac{4}{5}
\end{align*}
again using the fact that $f(n)$ and $g(n)$ are negligible compared to $n$. Thus, for sufficiently large $n$, $\Lambda_1(\alpha) \leq 9/10$, completing the proof.  
\end{proof}

\begin{lemma}\label{ksmall}
Let $\alpha = (\lambda_1, \mu_1, \lambda_2)$ be a $\vec{b}$-partition, where $\vec{b} = \vec{b}_n(f, g)$, such that $\left| \beta^{Re} \right| = (i, j, k)$. Then, 
\begin{equation*}
\Lambda_1(\alpha) \leq 1 - s_1(i) + s_2(j) -  s_3(k)
\end{equation*}
where 
\begin{align*}
s_1(i) &= \begin{cases} \frac{2i(f(n)-i+1)}{n+2\Delta} & i < \frac{f(n)}{2}\\  
              \frac{if(n)}{n+2\Delta}  & i \geq  \frac{f(n)}{2}  
\end{cases} \\
s_2(j) &= \frac{2j(f(n) - g(n) - j +1)}{n+2\Delta}\\
s_3(k) &= \frac{2k(n - g(n) - k + 1)}{n+2\Delta}
\end{align*}

\end{lemma}

\begin{proof}[\bf Proof:] From Equation \eqref{lambdaexpression},
\begin{equation*}
\Lambda(\alpha) = C(\lambda_1) - C(\mu_1) + C(\lambda_2)
\end{equation*}
Furthemore, from Lemma \ref{partitionbound},
\begin{equation*}
C(\lambda_1) \leq 
\begin{cases} 
\frac{f(n)^2 - f(n)}{2} - i(f(n) -i +1) & i < \frac{f(n)}{2}   \\
\frac{f(n)^2 - f(n)}{2} - \frac{if(n)}{2} & i \geq \frac{f(n)}{2}
\end{cases}
\end{equation*}
Thus, from the definition of $s_1(i)$, 
\begin{equation*}
C(\lambda_1)  \leq \frac{f(n)^2 - f(n)}{2} - \frac{n+2\Delta}{2}s_1(i) 
\end{equation*}
Using Lemma \ref{skewpartitionbound} and simplifying, 
\begin{align*}
C(\lambda_2) - C(\mu_1) \leq \Delta  -\frac{ f(n)^2 - f(n)}{2}  - \frac{n+2\Delta}{2}s_2(j) + \frac{n+2\Delta}{2}s_3(k)
\end{align*}
using the fact that $2\Delta = n^2 - 2ng(n) + 2f(n)g(n) - n$. 

Thus, combining the two, 
\begin{align*}
\Lambda(\alpha) = C(\lambda_1) - C(\mu_1) + C(\lambda_2) \leq \Delta - \frac{n+2\Delta}{2}\left( s_1(i) - s_2(j) + s_3(k)\right)
\end{align*}
Since $\Lambda_1(\alpha) = \frac{n+2\Lambda(\alpha)}{n+2\Delta}$,
\begin{equation*}
\Lambda_1(\alpha)  \leq 1 - s_1(i) + s_2(j) - s_3(j)
\end{equation*} 
as required. 
\end{proof}

\section{Lead Term Analysis and Chi-Squared Lower Bound}\label{sectchisquaredlower}

The eigenvalue bounds derived above will now be used to provide some lead term analysis and also to prove the lower bound of Theorem \ref{chitheorem}. This will use Lemmas \ref{maxskewpartition} and \ref{partitionbound} -- Lemma \ref{ksmall} could also be used, but the previous lemmas are more hands on. 

Let $\alpha = (\lambda_1, \mu_1, \lambda_2)$ be a $\vec{b}$-partition, where $\vec{b} = \vec{b}_n(f, g)$. Then, 
\begin{equation*}
\Lambda(\alpha) = C(\lambda_1) - C(\mu_1) + C(\lambda_2),
\end{equation*}
Lemmas \ref{maxskewpartition} and \ref{partitionbound} show that $C(\lambda_1) - C(\mu_1)$ is maximized at $\lambda_1 = (f(n))$ and $\mu_1 = (f(n) - g(n))$, whereas $C(\lambda_2)$ is maximized at $(n-g(n))$. Thus, letting 
\begin{equation*}
\alpha_0 = ((f(n)), (f(n) - g(n)), (n-g(n))),
\end{equation*}
$\Lambda(\alpha_0)$ must be the maximal eigenvalue of $U$. Indeed, doing the calculation gives
\begin{align*}
\Lambda(\alpha_0) &= \frac{n^2 - n - 2ng(n) + f(n)g(n)}{2} = \Delta
\end{align*}
which is precisely the expected maximal eigenvalue of an adjacency matrix of a regular graph with degree $\Delta$. Intuitively, the next highest eigenvalues should be $\Lambda(\alpha)$ for $\alpha$ close to $\alpha_0$. If $\alpha = (\lambda_1, \mu_1, \lambda_2) \neq \alpha_0$, clearly either $\lambda_1 \neq (f(n))$ or $\lambda_2 \neq (n- g(n))$. Thus, the obvious candidates for next highest eigenvalue are  
\begin{align*}
&\alpha_1 = ((f(n)-1, 1), (f(n) - g(n)), (n - g(n))) \text{ and } \\
&\alpha_2 = ((f(n)), (f(n) - g(n)), (n - g(n)-1, 1)) 
\end{align*}
Indeed, if $\alpha = (\lambda_1, \mu_1, \lambda_2)$ and $\lambda_1 \neq (f(n))$, then from Lemmas \ref{partitionbound} and \ref{maxskewpartition} it is easy to conclude that
\begin{align*}
\Lambda(\alpha) &= C(\lambda_1) - C(\mu_1) + C(\lambda_2)\\
       &\leq \frac{f(n)^2 - f(n)}{2} - f(n) + (C(\lambda_2) - C(\mu_1))\\
   & \leq C(f(n)) - f(n)  - C(f(n) - g(n)) + C(n - g(n))\\
            &\leq \Lambda(\alpha_1)
\end{align*}
and similarly, if $\lambda_2 \neq (n - g(n))$,  
\begin{align*}
\Lambda(\alpha) &\leq C(f(n))  - C(f(n) - g(n)) + C(n - g(n)) - (n - g(n)) \\
                &\leq \Lambda(\alpha_2)
\end{align*}
The above arguments (or a straighforward calculation) should make it clear that 
\begin{align}\label{nexthighesteigs}
\Lambda(\alpha_1) = \Delta - f(n) \ \text{ and } \
\Lambda(\alpha_2) = \Delta - (n -g(n))
\end{align}
The following steps derive the bounds corresponding to $\alpha_1$ and $\alpha_2$ for the lead-term analysis. 
\\

\noindent \textbf{Bound corresponding to $\alpha_1$:} 

\noindent Here, the term $\dim(R_n(\alpha_1)) \Lambda_1(\alpha_1)^{2t}$ is used, where
\begin{equation}\label{alpha1}
\alpha_1 = ((f(n)-1, 1), (f(n) - g(n)), (n - g(n)))
\end{equation} 
as defined above. 
Now, from Equation \ref{dimeig}, 
\begin{equation*}
\dim(R_n(\alpha_1)) = \left|X_{\lambda_1}\right|\left|X_{\lambda_1/\mu_1}\right|
\left|X_{\lambda_2/\mu_1}\right|
\left|X_{\lambda_2}\right|
\end{equation*}
As noted in Remark \ref{degskewrep}, the degree of a skew representation is just the number of standard Young tableaux of that skew shape. Thus, 
\begin{equation*}
\left| X_{\lambda_1}\right| = \left| X_{(f(n)-1, 1)}\right| = (f(n) -1)
\end{equation*}
since choosing a standard Young tableau for $(f(n)-1, 1)$ just requires picking a number other than $1$ for the second row -- the numbers in the first row have to be ordered, and and picking $1$ to go in the second row would result in a contradiction in the first column. 

Now, $\lambda_1/\mu_1 = (f(n)-1, 1)/(f(n) - g(n))$, which consists of precisely $g(n) -1$ squares starting at $f(n) - g(n)+1$ in the first row, and one square in the second row. Since the square in the second row is not directly below any square in the first row, choosing a standard Young tableau just requires picking any number for the second row. Thus, 
\begin{equation*}
\left| X_{\lambda_1/\mu_1}\right| = g(n)
\end{equation*}
Since $\lambda_2/\mu_1$ and $\lambda_2$ are both just single rows, \begin{equation*}
\left| X_{\lambda_2/\mu_1}\right| = \left| X_{\lambda_2} \right| = 1
\end{equation*}
Combining,  
\begin{equation*}
\dim(R_n(\alpha_1)) = (f(n)-1)g(n).
\end{equation*}
From Equation \eqref{nexthighesteigs},
\begin{equation*}
\Lambda_1(\alpha_1) = \frac{n +2\Lambda(\alpha_1)}{n+2\Delta} = 1 - \frac{2f(n)}{n+2\Delta}
\end{equation*}
and hence 
\begin{align} \label{alpha1bound}
\dim(R_n(\alpha_1)) \Lambda_1(\alpha_1)^{2t} &= (f(n)-1)g(n) \left( 1 - \frac{2f(n)}{n + 2\Delta}\right)^{2t} \\
 & \approx (f(n)-1)g(n) \exp \left( - \frac{4tf(n)}{n + 2\Delta} \right) \nonumber
\end{align}
Thus, making this lead term at most $e^{-c}$ for some constant $c$ requires
\begin{equation*}
\frac{4t f(n)}{n+2\Delta} \geq \log (f(n)-1) + \log g(n) + c 
\end{equation*}
so it clearly suffices to have
\begin{align}\label{tfirstbound}
 t &\geq \frac{(n+2\Delta)(\log f(n) + \log g(n))}{4f(n)} + c\frac{n+2\Delta}{4f(n)}
\end{align} 
Thus, the above bound is the contribution of $\alpha_1$. Turn next to $\alpha_2$. 
\\

\noindent \textbf{Bound corresponding to $\alpha_2$:} 

\noindent The case of 
\begin{equation*}
\alpha_2 = ((f(n)), (f(n) - g(n)), (n - g(n),1))
\end{equation*} 
is entirely analogous. Identical calculations show that 
\begin{equation*}
\dim(R_n(\alpha_2)) = (n - g(n)-1)(n-f(n))
\end{equation*}
and that
\begin{equation*}
\Lambda_1(\alpha_2) = 1 - \frac{2(n - g(n))}{n+2\Delta}
\end{equation*}
Hence, 
\begin{align*}
\dim(R_n(\alpha_2)) \Lambda_2(\alpha_1)^{2t}  \approx (n - g(n)-1)(n-f(n)) \exp \left( - \frac{4t(n  -g(n))}{n + 2\Delta} \right)
\end{align*}
giving the bound
\begin{align}\label{tsecondbound}
 t &\geq \frac{(n+2\Delta)(\log (n-f(n)) + \log (n- g(n)))}{4(n-g(n))} + c\frac{n+2\Delta}{4( n - g(n))}
\end{align} 
To compare the bounds, first prove the following simple lemma: 

\begin{lemma}\label{boundcompare}
If $(x_n)$ and $(y_n)$ are positive sequences such that $\lim_{n\rightarrow \infty} \frac{x_n}{y_n} = 0$, and $\lim_{n\rightarrow \infty} x_n = \infty$, then for an arbitrarily large constant $c$,
\begin{equation*}
\frac{\log x_n}{x_n} \geq c \frac{ \log y_n}{y_n}
\end{equation*}
for sufficiently large $n$. 
\end{lemma}
\begin{proof}[\bf Proof:]
Rewriting, 
\begin{align*}
\frac{\log y_n}{y_n} &= \frac{\log (y_n/x_n)+ \log x_n}{x_n}\cdot\frac{x_n}{y_n} \\
&=\frac{1}{x_n} \left( - \frac{x_n}{y_n} \log \left(\frac{x_n}{y_n}\right) + \frac{x_n}{y_n}\cdot \log x_n \right)
\end{align*}
Since $\frac{x_n}{y_n}$ approaches $0$, and $x_n$ approaches $\infty$,  $ - \frac{x_n}{y_n} \log \left(\frac{x_n}{y_n}\right)$ approaches $0$ and thus is eventually less than $\frac{1}{2c}\log x_n$. Similarly, $\frac{x_n}{y_n} \log x_n$ is eventually less than $\frac{1}{2c} \log x_n$. Combining, for sufficiently large $n$, 
\begin{equation*}
\frac{\log y_n}{y_n} \leq \frac{1}{c} \frac{\log x_n}{x_n}
\end{equation*}
as required. 
\end{proof}

One of the assumptions in Theorem $\ref{chitheorem}$ is that $f(n)/n$ approaches $0$. Since by definition $g(n) \leq f(n)$, the bound in Equation \eqref{tsecondbound} is of order $(n+2\Delta) \log n/n$, whereas the bound in Equation \eqref{tfirstbound} is of order $(n+2\Delta) \log f(n)/f(n)$. The above lemma with $x_n = f(n)$ and $y_n =n$ makes it clear that Equation \eqref{tfirstbound} is the stronger bound. This lead term analysis suggests chi-squared cutoff around 
\begin{equation}
t = \frac{(n+2\Delta)(\log f(n) + \log g(n))}{4f(n)} 
\end{equation}
with a window of order $\frac{n+2\Delta}{4f(n)}$. 

Having derived the second highest eigenvalue, the lower bound part of Theorem \ref{chitheorem} can be proved. If $t$ is defined as 
\begin{equation}\label{tlowerboundrestatement}
t = \frac{(n+2\Delta)(\log f(n) + \log g(n))}{4f(n)} - c \frac{n + 2\Delta}{4f(n)}
\end{equation}
it is shown that
\begin{equation*}
\left\| P^t(x, \cdot) - \pi \right\|_{2, \pi} \geq \frac{1}{2} e^{c/2}
\end{equation*}
for sufficiently large $n$. 

\begin{proof}[\bf Proof of Lower Bound in Theorem \ref{chitheorem}:]
Let $\alpha_1 = ((f(n)-1, 1), (f(n) - g(n)), (n - g(n)))
$ as in Equation \eqref{alpha1}. From Equation \eqref{alleigsequation}, 
\begin{align*}
 \left\| P^t(x, \cdot) - \pi \right\|_{2, \pi} &= \sqrt{\sum_{\alpha} \dim(R_n(\alpha)) \Lambda_1(\alpha)^{2t}} \\
&\geq \sqrt{\dim(R_n(\alpha_1)) \Lambda_1(\alpha_1)^{2t}}
\end{align*}
Furthermore, from Equation \eqref{alpha1bound},  
\begin{equation*}
\dim(R_n(\alpha_1)) \Lambda_1(\alpha_1)^{2t} = (f(n)-1)g(n) \left( 1 - \frac{2f(n)}{n + 2\Delta}\right)^{2t} 
\end{equation*} 
Since $\frac{f(n)}{n} \rightarrow 0$, $\frac{f(n)}{n+2\Delta} \rightarrow 0$ as $n$ approaches $\infty$. Thus, for $t$ as defined in Part 2 of Theorem \ref{chitheorem} and as restated above in Equation \eqref{tlowerboundrestatement},
\begin{align*}
\lim_{n\rightarrow \infty} (f(n)-1)g(n) \left( 1 - \frac{2f(n)}{n+2\Delta}\right)^{2t}& = \lim_{n\rightarrow \infty} (f(n)-1)g(n)\exp\left(-\frac{4tf(n)}{n+2\Delta} \right) \\
&=\lim_{n\rightarrow \infty} (f(n)-1)g(n)e^{-\log f(n) - \log g(n) + c}\\
&= \lim_{n\rightarrow \infty} \frac{f(n) - 1}{f(n)} e^{c} = e^{c}
\end{align*}
since $f(n) \rightarrow \infty$. Combining the equations above,  
\begin{equation*}
\left\| P^t(x, \cdot) - \pi \right\|_{2, \pi} \geq \sqrt{ \dim (R_n(\alpha_1)) \Lambda_1(\alpha_1)^{2t}} \longrightarrow e^{c/2} \text{ as $n \rightarrow \infty$}
\end{equation*}
and therefore, for sufficiently large $n$, the chi-squared distance is at least $\frac{e^{c/2}}{2}$. This proves the lower bound. 
\end{proof}
\begin{rmk}
Note that while the assumptions of Theorem \ref{chitheorem} are used in the above proof to simplify computation, their full strength is not needed. Indeed, a similar lower bound can be derived for almost any functions $f$ and $g$. However, if $n - g(n) \leq f(n)$ then this lower bound will not correspond to the second-highest eigenvalue, as that will be associated to the $\alpha_2$ defined above. 
\end{rmk}

\section{Dimensions of Eigenspaces}\label{dimensionofeigenspacecalc}
This section continues with the proof of the the upper bound in Theorem \ref{chitheorem}. Here is a recap of the bounds in Section \ref{boundingtheeigenvalues}. 

\begin{defn}\label{sijk}
Let $(i, j, k)$ be a triple of integers. Then, define 
\begin{equation}\label{sdef}
s(i, j, k) = 
\begin{cases} \frac{9}{10} & k \geq (n -g(n))/5\\
             1 - s_1(i) + s_2(j) - s_3(k) & \text{otherwise} \end{cases}
\end{equation}
where $s_1, s_2$ and $s_3$ are defined as in Lemma \ref{ksmall} -- that is, 
\begin{align*}
s_1(i) &= \begin{cases} \frac{2i(f(n)-i+1)}{n+2\Delta} & i < \frac{f(n)}{2}\\  
              \frac{if(n)}{n+2\Delta}  & i \geq  \frac{f(n)}{2}  
\end{cases} \\
s_2(j) &= \frac{2j(f(n) - g(n) - j +1)}{n+2\Delta}\\
s_3(k) &= \frac{2k(n - g(n) - k + 1)}{n+2\Delta}
\end{align*} 
\end{defn}
\begin{rmk}
The above $s(i, j, k)$ are chosen to simplify notation. Combining Lemmas \ref{kbig} and \ref{ksmall},  
\begin{equation}
\Lambda_1(\alpha) \leq s\left( \left| \alpha^{Re} \right| \right)
\end{equation}
which will clearly be useful. 
\end{rmk}

Here is a sketch out the rest of the proof. By Equation \eqref{positiveeigs}, the quantity to be bounded is
\begin{equation*}
\sum_{1 \neq \Lambda_1(\alpha) \geq 0} \dim(R_n(\alpha)) \Lambda_1(\alpha)^{2t}
\end{equation*}
Since the upper bound above depends purely on $\left| \alpha^{Re} \right|$,  rearrange the above quantity by the value of $\left| \alpha^{Re} \right|$. From the previous section, it is clear that $\left|\alpha^{Re} \right| = (0,0,0)$ corresponds to $\Lambda_1(\alpha) = 1$. Thus, 
\begin{align}\label{decompositionbyijk}
\sum_{1 \neq \Lambda_1(\alpha) \geq 0} \dim(R_n(\alpha)) \Lambda_1(\alpha)^{2t}   
 &\leq \sum_{\left|\alpha^{Re} \right| \neq (0,0,0)} \dim (R_n(\alpha)) s\left( \left| \alpha^{Re} \right| \right)^{2t} \nonumber \\
 &=  \sum_{(i, j, k) \neq (0, 0, 0)} s(i, j, k)^{2t}\sum_{\left| \alpha^{Re} \right| = (i, j, k)} \dim (R_n(\alpha)) 
\end{align}

The proof below will be organized as follows. The current section finds an expression for the sum
\begin{equation*}
\sum_{\left| \alpha^{Re} \right| = (i, j, k)} \dim (R_n(\alpha)) 
\end{equation*}
in terms of $i, j,$ and $k$. This leaves the (rather unwieldy) sum on the right-hand side of Equation \eqref{decompositionbyijk} in terms of the three indices $i, j,$ and $k$. In Section \ref{chisquaredupperchapter}, this sum is split into a number of pieces that depend on the precise values of the indices, and supporting lemmas are proved for the size of each piece. All the bounds are then combined into a proof of the upper bound part of Theorem \ref{chitheorem}. 

The current section is devoted to proving the following lemma:

\begin{lemma}\label{dimensionsum}
Assume that $i \geq j \leq k$. Then, 
\begin{equation*}
\sum_{\left| \alpha^{Re} \right| = (i, j, k)} \dim (R_n(\alpha)) \leq  {f(n) \choose i}{g(n) \choose  i-j} {n- f(n) \choose k-j}{n - g(n) \choose k} \frac{i!k!}{j!}
\end{equation*} 
\end{lemma} 
\begin{rmk}
Note that $\left| \alpha^{Re} \right|$ is defined to be $\left(\left| \lambda_1^{Re} \right|, \left| \mu_1^{Re} \right|,\left| \lambda_2^{Re} \right|\right)$, where $\lambda_1 \supseteq \mu_1 \subseteq \lambda_2$. Thus, if
\begin{equation*}
\left| \alpha^{Re} \right| = (i,j, k)
\end{equation*}
then $i \geq j \leq k$. Therefore, the the condition in Lemma \ref{dimensionsum} is the natural one. 
\end{rmk}

\begin{defn}\label{aijk}
To simplify notation from now on, define 
\begin{equation*}
a(i, j, k) = {f(n) \choose i}{g(n) \choose  i-j} {n- f(n) \choose k-j}{n - g(n) \choose k} 
\end{equation*}
The $n$ in the above expression will always be implied. 
\end{defn}

A number of supporting lemmas will eventually yield Lemma \ref{dimensionsum} above. 
\begin{lemma}\label{dimupperbound}
Let $\alpha= (\lambda_1, \mu_1, \lambda_2)$ be a $\vec{b}$-partition, where $\vec{b} = \vec{b}_n(f, g)$, satisfying $\left|\alpha^{Re}\right| = (i, j, k)$. Then, 
\begin{equation}
\dim(R_n(\alpha)) \leq a(i, j, k)
 \left| X_{\lambda_1^{Re}} \right| \left|X_{\lambda_1^{Re}/\mu_1^{Re}} \right| \left|X_{\lambda_2 ^{Re}/\mu_1^{Re}} \right|\left|X_{\lambda_2^{Re}} \right| 
\end{equation}
\end{lemma} 
\begin{proof}[\bf Proof:]
From Equation \eqref{dimeig} in Theorem \ref{hanlonthm}, 
\begin{equation*}
 \dim(R_n(\alpha)) = \left|X_{\lambda_1} \right| \left|X_{\lambda_1 /\mu_1}\right| \left|X_{\lambda_2} \right| \left|X_{\lambda_2 /\mu_1}\right|
\end{equation*}

Furthermore, as noted in Remark \ref{degskewrep}, $\left| X_{\alpha/\beta}\right|$ is the number of standard Young tableaux of shape $\alpha/\beta$. Use this to examine the expression above. Clearly, $\left| X_{\lambda_1} \right|$ is the number of standard Young tableaux of shape $\lambda_1$. A naive way of trying to construct such a tableau would be to pick the $f(n) - i$ elements of $\{1, 2, \dots, f(n)\}$ that will go in the first row of the standard Young tableau of shape $\lambda_1$, and then use the remaining $i$ elements to construct a `shifted' standard Young tableau of the remainder of $\lambda_1$, called $\lambda_1^{Re}$ above. Note that once the elements of the first row are chosen, they must be arranged in exactly increasing order: thus, each choice of subset and of remaining `shifted' tableau results in exactly one Young tableau. While this will vastly overcount the number of standard Young tableaux of shape $\lambda_1$, every single one can be constructed in this way. Thus,
\begin{equation*}
\left| X_{\lambda_1} \right| \leq {f(n) \choose {f(n) - i}} \left| X_{\lambda_1^{Re}} \right| = {f(n) \choose i} \left| X_{\lambda_1^{Re}} \right| 
\end{equation*}
Similarly, to construct a skew tableau of shape $\lambda_1/\mu_1$, pick the $g(n) - (i - j)$ elements for the first row, and construct a `shifted' tableau of shape $\lambda_1^{Re}/\mu_1^{Re}$ with the remaining numbers. Thus,
\begin{equation*}
\left|X_{\lambda_1 /\mu_1}\right| \leq { g(n) \choose  i-j} \left|X_{\lambda_1 ^{Re}/\mu_1^{Re}} \right| 
\end{equation*}
Proceeding in this way, 
\begin{align*}
\left|X_{\lambda_2 /\mu_1}\right| &\leq {n- f(n) \choose k-j} \left|X_{\lambda_2 ^{Re}/\mu_1^{Re}} \right| \\
\left|X_{\lambda_2}\right| &\leq {n - g(n) \choose k} \left|X_{\lambda_2^{Re}} \right|.
\end{align*}
Multiplying the above inequalitites together gives precisely 
\begin{equation*}
\dim(R_n(\alpha)) \leq a(i, j, k)
 \left| X_{\lambda_1^{Re}} \right| \left|X_{\lambda_1^{Re}/\mu_1^{Re}} \right| \left|X_{\lambda_2 ^{Re}/\mu_1^{Re}} \right|\left|X_{\lambda_2^{Re}} \right| 
\end{equation*}
as desired. 
\end{proof}
Using the lemma above, 
\begin{equation*}
\begin{split}
\sum_{\left| \alpha^{Re} \right| = (i, j, k)}  \dim &(R_n(\alpha)) \\
& \leq a(i, j, k)\sum_{\left| \alpha^{Re} \right| = (i, j, k)}\left| X_{\lambda_1^{Re}} \right| \left|X_{\lambda_1^{Re}/\mu_1^{Re}} \right| \left|X_{\lambda_2 ^{Re}/\mu_1^{Re}} \right|\left|X_{\lambda_2^{Re}} \right| 
\end{split}
\end{equation*}
In order to simplify things slightly, the following simple lemma is useful: 

\begin{lemma}\label{alpharesum}
For any $i, j, k$, 
\begin{equation*}
\begin{split}
\sum_{\left|\alpha^{Re}\right| = (i, j, k)} \left| X_{\lambda_1^{Re}} \right| \left|X_{\lambda_1^{Re}/\mu_1^{Re}} \right| & \left|X_{\lambda_2 ^{Re}/\mu_1^{Re}} \right|\left|X_{\lambda_2^{Re}} \right|  \\
&\leq \sum_{(\lambda_1', \mu_1', \lambda_2') } \left| X_{\lambda_1'} \right| \left|X_{\lambda_1'/\mu_1'} \right| \left|X_{\lambda_2'/\mu_1'} \right|\left|X_{\lambda_2'} \right|
\end{split}
\end{equation*}
where the right-hand sum is over triples of partitions such that $\left| \lambda_1' \right| = i, \left| \mu_1' \right| =j, \left|\lambda_2'\right| = k$, and $\lambda_1' \supseteq \mu_1' \subseteq \lambda_2'$.
\end{lemma}
\begin{proof}[\bf Proof:]
Clearly, if $\left|\alpha^{Re}\right| = (i, j, k)$, then 
\begin{equation*}
\left| \lambda_1^{Re} \right| = i, \left| \mu_1^{Re} \right| =j, \left|\lambda_2^{Re}\right| = k, \text{ and } \lambda_1^{Re} \supseteq \mu_1^{Re} \subseteq \lambda_2^{Re}
\end{equation*}
Furthermore, given $\vec{b} = \vec{b}_n(f, g)$, $\alpha^{Re}$ uniquely determines $\alpha$ (although not every choice of $(\lambda_1', \mu_1', \lambda_2')$ will actually produce an $\alpha$). Hence, the lemma follows trivially. 
\end{proof}
Now, note that the sum on the right-hand side of the above lemma looks a lot like a sum of $\dim(R_n(\alpha'))$ over $\vec{a}$-partitions $\alpha'$ for some $\vec{a}$. Using this heuristic, consider rewriting the above as 
\begin{equation*}
\sum_{\alpha'} \dim(R_n(\alpha'))  = \left|S_{M(\vec{a})}\right|
\end{equation*}
where the equality follows since the sum of the dimensions of the eigenspaces is just the dimension of the whole space.

However, the above heuristic has the following trouble: if either $i=j$ or $j=k$, then the $(\lambda_1', \mu_1', \lambda_2')$ above are not in fact $\vec{a}$-partitions for any $\vec{a}$, since $\lambda_1' = \mu_1'$ or $\mu_1' = \lambda_2'$ are forbidden. However, this turns out to be a non-essential part of the definition of $\vec{a}$-partitions. The following supporting lemma overcomes the issue. 

\begin{lemma}\label{partitionformula}
Let $n \geq m$, and let $\lambda$ be a partition of $n$. Then, the following equality holds
\begin{equation*}
\left|X_\lambda\right| = \sum_{\mu \leq \lambda,\\ |\mu| = m} \left|X_{\lambda/\mu}\right| \left| X_\mu \right|
\end{equation*}
\end{lemma}
\begin{proof}[\bf Proof:]
This may well be a well-known formula. However, it has a simple combinatorial proof presented below. As noted above, $\left| X_\lambda\right|$ is the number of standard Young tableaux of shape $\lambda$, and similarly for $\mu$ and $\lambda/\mu$. Let $Tab(s)$ denote the set of standard Young tableaux of shape $s$, whether $s$ is a partition or a skew-partition. Then, what is needed is a bijection $g$ such that 
\begin{equation}\label{bij}
g: \bigcup_{|\mu| = m} Tab(\mu)\times Tab(\lambda/\mu) \longrightarrow Tab(\lambda)
\end{equation}
where the union is over partitions $\mu$. 

To define the bijection $g$, note that if $Y_\lambda$ is a standard Young tableau of shape $\lambda$ and $Y_{\lambda/\mu}$ is a Young tableau of shape $\lambda/\mu$, a Young tableau of shape $\lambda$ can be created by adding $m$ to every entry of the $Y_{\lambda/\mu}$ and then sticking the two tableaux together. Call this new tableau $g(Y_\mu, Y_{\lambda/\mu})$. 

For example, if 
$\lambda = (4,2,1)$, and $\mu = (3,1)$, then the following Young tableaux

\begin{picture}(250,60)(20,-5)

\put(85, 35){$Y_{\mu} =$}
\linethickness{0.6 pt}
\put(110, 45){\line(1, 0){45}}
\put(110, 30){\line(1, 0){45}}
\put(110, 15){\line(1,0){15}}

\put(110, 15){\line(0,1){30}}
\put(125, 15){\line(0,1){30}}
\put(140, 30){\line(0,1){15}}
\put(155, 30){\line(0,1){15}}

\put(115, 35){1}
\put(130, 35){3}
\put(145, 35){4}
\put(115, 20){2}

\put(210, 35){$Y_{\lambda/\mu} =$}
\linethickness{0.6 pt}
\put(285, 45){\line(1, 0){15}}
\put(285, 30){\line(1, 0){15}}
\put(255, 30){\line(1, 0){15}}
\put(240, 15){\line(1,0){30}}
\put(240, 0){\line(1,0){15}}

\put(300, 30){\line(0,1){15}}
\put(285, 30){\line(0,1){15}}
\put(255, 0){\line(0,1){30}}
\put(240, 0){\line(0,1){15}}
\put(270, 15){\line(0,1){15}}

\put(290, 35){1}
\put(260, 20){3}
\put(245, 5){2}

\end{picture}
\\ can be combined into $g(Y_\mu, Y_{\lambda/\mu})$, which is

\begin{picture}(250,60)(5,-5)

\linethickness{0.6 pt}
\put(125, 35){$Y_\lambda = $}

\put(150, 0){\line(0,1){45}}
\put(150, 45){\line(1, 0){60}}
\put(165, 0){\line(0,1){45}}
\put(180, 15){\line(0,1){30}}
\put(195, 30){\line(0,1){15}}
\put(210, 30){\line(0,1){15}}
\put(150, 30){\line(1, 0){60}}
\put(150, 15){\line(1,0){30}}
\put(150, 0){\line(1, 0){15}}

\put(155, 35){1}
\put(170, 35){3}
\put(185, 35){4}
\put(200, 35){5}
\put(155, 20){2}
\put(170, 20){7}
\put(155, 5){6}

\end{picture}

Clearly, for any $Y_\mu \in Tab(\mu)$ and $Y_{\lambda/\mu} \in Tab(\lambda/\mu)$, the above bijection results in a tableau of shape $\lambda$ which is filled with the numbers $\{1, 2, \dots, n\}$. Thus, it remains to check that the rows and columns in $g(Y_\mu, Y_{\lambda/\mu})$ are in increasing order from left to right and from top to bottom. 

Proceed by contradiction: assume there are $i$ and $j$ such that $i < j$, but $j$ is strictly to the left of $i$, or strictly above $i$ in $g(Y_\mu, Y_{\lambda/\mu})$. If $i$ and $j$ are both at most $m$, then they both appeared in $Y_\mu$, and this is impossible; similarly, if both $i$ and $j$ are greater than $m$, then $i-m$ and $j-m$ both appeared in $Y_{\lambda/\mu}$ and this is similarly impossible. Thus, assume that $i \leq m < j.$ But then $j$ must be in a square belonging to $\lambda/\mu$, and $i$ must be in a square belonging to $\mu$, and therefore it is impossible for $j$ to be strictly to the left or strictly above $i$. Thus, $g(Y_{\mu}, Y_{\lambda/\mu})$ is a standard Young tableau, so the map $g$ is well-defined. 

Now for the inverse $f$ of the map $g$. It is easy to see from above that $\mu$ is defined precisely by the set of squares which contain the numbers $\{1, 2, \dots, m\}$. Using arguments identical to the above, in any standard Young tableau $Y_{\lambda}$ the set of squares containing the elements $\{1, 2, \dots, m\}$ is a partition of $m$. Thus, the inverse map $f$ must map $Y_\lambda$ to $(Y_{\mu}, Y_{\lambda/\mu})$, where $Y_{\mu}$ is simply the standard Young tableaux induced by the squares containing $\{1, 2, \dots, m\}$, and $Y_{\lambda/\mu}$ is obtained by deleting the squares in $\mu$ from $Y_{\lambda}$, and subtracting $m$ from the remaining squares. For example, if 

\begin{picture}(250,60)(5,-5)

\linethickness{0.6 pt}
\put(125, 35){$Y_\lambda = $}

\put(150, 0){\line(0,1){45}}
\put(150, 45){\line(1, 0){60}}
\put(165, 0){\line(0,1){45}}
\put(180, 15){\line(0,1){30}}
\put(195, 30){\line(0,1){15}}
\put(210, 30){\line(0,1){15}}
\put(150, 30){\line(1, 0){60}}
\put(150, 15){\line(1,0){30}}
\put(150, 0){\line(1, 0){15}}

\put(155, 35){1}
\put(170, 35){2}
\put(185, 35){5}
\put(200, 35){6}
\put(155, 20){3}
\put(170, 20){4}
\put(155, 5){7}

\end{picture}
\\then the inverse map $f$ maps it to the pair

\begin{picture}(250,60)(20,-5)

\put(85, 35){$Y_{\mu} =$}
\linethickness{0.6 pt}
\put(110, 45){\line(1, 0){30}}
\put(110, 30){\line(1, 0){30}}
\put(110, 15){\line(1,0){30}}

\put(110, 15){\line(0,1){30}}
\put(125, 15){\line(0,1){30}}
\put(140, 15){\line(0,1){30}}

\put(115, 35){1}
\put(130, 35){2}
\put(115, 20){3}
\put(130, 20){4}

\put(210, 35){$Y_{\lambda/\mu} =$}
\linethickness{0.6 pt}
\put(270, 45){\line(1, 0){30}}
\put(270, 30){\line(1, 0){30}}
\put(240, 15){\line(1,0){15}}
\put(240, 0){\line(1,0){15}}

\put(300, 30){\line(0,1){15}}
\put(285, 30){\line(0,1){15}}
\put(270, 30){\line(0,1){15}}
\put(240, 0){\line(0,1){15}}
\put(255, 0){\line(0,1){15}}

\put(275, 35){1}
\put(290, 35){2}
\put(245, 5){3}

 \end{picture}
\\and hence the bijection has an explicit inverse. This shows that the sizes of the sets in Equation \eqref{bij} are equal, and hence 
\begin{equation*}
\left|X_\lambda\right| = \sum_{\mu \leq \lambda,\\ |\mu| = m} \left|X_{\lambda/\mu}\right| \left| X_\mu \right|
\end{equation*}
as required. 
\end{proof}

The next lemmas tackle the expression on the right-hand side of Lemma \ref{alpharesum}.

\begin{lemma}\label{sumdimlemma}
Let $(i, j, k)$ be a triple of positive integers that satisifes $i \geq j \leq k$. Define
\begin{equation}\label{bigquant}
A  = \sum_{(\lambda_1', \mu_1', \lambda_2')}  \left| X_{\lambda_1'}\right| \left|X_{\lambda_1'/\mu_1'}\right|\left| X_{\lambda_2' /\mu_1'} \right| \left| X_{\lambda_2'} \right| 
\end{equation}
where the sum is over triples of partitions $(\lambda_1', \mu_1', \lambda_2')$ such that $\left| \lambda_1' \right| = i, \left| \mu_1' \right| = j, \left| \lambda_2' \right| = k$ and $\lambda_1' \supseteq \mu_1' \subseteq \lambda_2'$. Then, 
\begin{equation*}
A = \frac{i! k!}{j!}
\end{equation*}
\end{lemma}

\begin{proof}[\bf Proof:] If $i > j < k$, define the vector $\vec{a}$ to be 
\begin{equation*}
\vec{a} = (1, 1, \dots, 1, i - j + 1, \dots, i - j+1)
\end{equation*}
where the number of initial $1$s is $i$, and the total number of terms in the vector is $i - j + k$. Then it is easy to check that the set of all $(\lambda_1', \mu_1', \lambda_2')$ such that $|\lambda_1'| = i, |\mu_1'| = j, |\lambda_2'| = k$ and $\lambda_1' \supseteq \mu_1' \subseteq \lambda_2'$ is the set of $\vec{a}$-partitions, as defined in Definition \ref{defbpartition}. Thus,
\begin{align*}
A &= \sum_{(\lambda_1', \mu_1', \lambda_2') \text{ $\vec{a}$-partition}}  \left| X_{\lambda_1'}\right| \left|X_{\lambda_1'/\mu_1'}\right|\left| X_{\lambda_2' /\mu_1'} \right| \left| X_{\lambda_2'} \right| \\
 &= \sum_{(\lambda_1', \mu_1', \lambda_2') \text{ $\vec{a}$-partition}} \dim(R_n(\lambda_1', \mu_1', \lambda_2'))
\end{align*}
where the second equality follows from Equation \eqref{dimeig} in Theorem \ref{hanlonthm}. But from the same theorem, 
\begin{equation*}
\mathbbm{C} S_{M(\vec{a})} = \bigoplus_{\alpha \text{ $\vec{a}$-partition}} R_n(\alpha)
\end{equation*}
Thus,  
\begin{equation*}
\sum_{\alpha \text{ $\vec{a}$-partition}} \dim(R_n(\alpha)) = \dim \mathbbm{C} S_{M(\vec{a})} = \left| S_{M(\vec{a})} \right|
\end{equation*}
But for the choice of $\vec{a}$, from Lemma \ref{sizelemma}, 
\begin{equation*}
\left| S_{M(\vec{a})} \right| = \frac{i! k!}{j!}
\end{equation*}
finally yielding that 
\begin{equation*}
A = \sum_{\alpha \text{ $\vec{a}$-partition}} \dim(R_n(\alpha)) =  \left| S_{M(\vec{a})} \right|  = \frac{i! k!}{j!}
\end{equation*}
as required. 

It remains to explain how to handle the case where $i = j$, or $j = k$, or both. Consider the case $i = j < k$ (the same method will apply to all the other cases.) As before, 
\begin{equation*}
 A = \sum_{(\lambda_1', \mu_1', \lambda_2')}  \left| X_{\lambda_1'}\right| \left|X_{\lambda_1'/\mu_1'}\right|\left| X_{\lambda_2' /\mu_1'} \right| \left| X_{\lambda_2'} \right|  
\end{equation*} 
where the sum is over $\left| \lambda_1' \right| = i, \left| \mu_1' \right| = j, \left| \lambda_2' \right| = k$ and $\lambda_1' \supseteq \mu_1' \subseteq' \lambda_2$. Since $i = j$, and $\lambda_1' \supseteq \mu_1'$, it must be that $\lambda_1' = \mu_1'$. Thus, 
\begin{equation*}
A = \sum_{(\lambda_1',  \lambda_2')}  \left| X_{\lambda_1'}\right| \left| X_{\lambda_2' /\lambda_1'} \right| \left| X_{\lambda_2'} \right|  
\end{equation*}
where the sum is over $|\lambda_1'| = i$, $\left| \lambda_2'\right| = k$, and $\lambda_1' \subseteq \lambda_2'$. But from Lemma \ref{partitionformula}, for a fixed $\lambda_2'$,
\begin{equation*}
\sum_{\lambda_1' \subseteq \lambda_2', |\lambda_1'| = i} \left| X_{\lambda_2' /\lambda_1'} \right| \left| X_{\lambda_1'} \right| = \left| X_{\lambda_2'}\right| 
\end{equation*}
Thus, 
\begin{align*}
A &= \sum_{(\lambda_1',  \lambda_2')}  \left| X_{\lambda_1'}\right| \left| X_{\lambda_2' /\lambda_1'} \right| \left| X_{\lambda_2'} \right| \\
 &= \sum_{|\lambda_2'| = k} \left| X_{\lambda_2'} \right| \sum_{\lambda_1' \subseteq \lambda_2', |\lambda_1'| = i}\left| X_{\lambda_2' /\lambda_1'} \right| \left| X_{\lambda_1'} \right| \\
 &= \sum_{\left| \lambda_2' \right| = k} \left| X_{\lambda_2'} \right|^2  
\end{align*}
and the sum on the right is well-known to be $k!$. Therefore, in this case, 
\begin{equation*}
A = k! = \frac{i! k!}{j!}
\end{equation*}
since $i = j$. The other cases with equality can be done similarly, completing the proof. 
\end{proof}
 
With all these preliminaries, the lemma from the beginning of the section can now be proved. 

\begin{proof}[\bf Proof of Lemma \ref{dimensionsum}]
From Lemma \ref{dimupperbound}, 
\begin{equation*}
\begin{split}
\sum_{\left|\alpha^{Re} \right| = (i, j, k)}\dim&(R_n(\alpha)) \\
&\leq a(i, j, k)
\sum_{\left|\alpha^{Re} \right| = (i, j, k)} \left| X_{\lambda_1^{Re}} \right| \left|X_{\lambda_1^{Re}/\mu_1^{Re}} \right| \left|X_{\lambda_2 ^{Re}/\mu_1^{Re}} \right|\left|X_{\lambda_2^{Re}} \right| 
\end{split}
\end{equation*}
Combining this with Lemma \ref{alpharesum},  
\begin{equation*}
\sum_{\left|\alpha^{Re} \right| = (i, j, k)}\dim(R_n(\alpha))\leq a(i, j, k) \sum_{(\lambda_1', \mu_1', \lambda_2') } \left| X_{\lambda_1'} \right| \left|X_{\lambda_1'/\mu_1'} \right| \left|X_{\lambda_2'/\mu_1'} \right|\left|X_{\lambda_2'} \right|
\end{equation*}
where the right-hand sum is over triples of partitions such that $\left| \lambda_1' \right| = i, \left| \mu_1' \right| =j, \left|\lambda_2'\right| = k$, and $\lambda_1' \supseteq \mu_1' \subseteq \lambda_2'$. However, from Lemma \ref{sumdimlemma},  
\begin{equation*}
\sum_{(\lambda_1', \mu_1', \lambda_2') } \left| X_{\lambda_1'} \right| \left|X_{\lambda_1'/\mu_1'} \right| \left|X_{\lambda_2'/\mu_1'} \right|\left|X_{\lambda_2'} \right| = \frac{i!k!}{j!}
\end{equation*}
and thus 
\begin{equation*}
\sum_{\left|\alpha^{Re} \right| = (i, j, k)}\dim(R_n(\alpha)) \leq a(i, j, k) \frac{i!k!}{j!}
\end{equation*}
as required. 
\end{proof}

\section{Chi-Squared Upper Bound}\label{chisquaredupperchapter}
This final section puts all the quantities together to prove the upper bound part of Theorem \ref{chitheorem}. For the remainder of this section, let $t$ be defined the way it is for the upper bound; that is, 
\begin{equation}\label{tupperboundrestated}
t = \frac{(n+2\Delta)(\log f(n) + \log g(n))}{4f(n)} + c \frac{n+2\Delta}{4f(n)}
\end{equation}
From Equation \eqref{decompositionbyijk},
\begin{equation}\label{sumtobound}
\sum_{1 \neq \Lambda_1(\alpha) \geq 0} \dim(R_n(\alpha)) \Lambda_1(\alpha)^{2t}   
 \leq \sum_{(i, j, k) \neq (0, 0, 0)} s(i, j, k)^{2t}\sum_{\left| \alpha^{Re} \right| = (i, j, k)} \dim (R_n(\alpha)) 
\end{equation}
where $s(i,j, k)$ is defined as in Definition \ref{sijk}. From the lead-term analysis in Section \ref{sectchisquaredlower}, the `limiting' eigenvalue corresponds to the case $(i, j, k) = (1, 0, 0)$. This suggests that term is the largest. 

The above sum will be broken up into various pieces and bounds will be proved for each piece. There are three zones: 

\begin{enumerate}
\item The first zone is $k \geq (n-g(n))/5$. As should be clear from the definition of $s(i, j, k)$, this is the zone with a constant upper bound for the eigenvalues $\Lambda_1(\alpha)$. This case is fairly straightforward, and will be done in Lemma \ref{klarge} below. 

\item The second zone is $i = 0$ and $k< (n-g(n))/5$. For this case, Lemma \ref{dimensionsum} is used for the upper bound. From the heuristics in the lead term analysis, the expected limiting term for this piece is $(0, 0, 1)$ -- since it was noted that this term imposes lower order restrictions than $(1,0,0)$, this case should also be fairly simple. This will be done in Lemma \ref{iiszero}. 

\item The final zone is $i > 0$ and $k< (n-g(n))/5$, for which Lemma \ref{dimensionsum} is again used. This is the case that contains the limiting term $(1, 0, 0)$, and as such should provide the biggest bound. This will be done in Lemma \ref{iisn'tzero}. 

\end{enumerate} 
To simplify notation, make the following definition: 

\begin{defn}\label{Apieces}
Define 
\begin{align*}
A_1 &= \left\{ (i, j, k) \left| \right. k\geq \frac{n-g(n)}{5} \right\}\\
A_2 &= \left\{ (i, j, k) \neq(0,0,0) \left| \right. k < \frac{n-g(n)}{5}, i = 0 \right\}\\
A_3 &= \left\{ (i, j, k) \left| \right. k < \frac{n-g(n)}{5}, i \neq 0 \right\} 
\end{align*}
This corresponds to the pieces being bounded.  
\end{defn}
\begin{lemma} \label{klarge} 
Using the definition of $A_1$ above, let 
\begin{equation*}
Q_1 = \sum_{(i, j, k) \in A_1} s(i, j, k)^{2t}\sum_{\left| \alpha^{Re} \right| = (i, j, k)} \dim (R_n(\alpha))  
\end{equation*}
Then, for sufficiently large $n$, $Q_1 \leq n^{-n}$. 
\end{lemma}
\begin{proof}[\bf Proof:]
By Definition \ref{sijk}, for $(i, j, k) \in A_1$, $s(i, j, k) = 9/10$. Thus, 
\begin{align*}
 Q_1 &\leq \sum_{(i, j, k) \in A_1} \left(\frac{9}{10}\right)^{2t} \sum_{\left| \alpha^{Re} \right| = (i, j, k)} \dim (R_n(\alpha)) \\
&\leq \left(\frac{9}{10}\right)^{2t} \sum_{\alpha} \dim (R_n(\alpha))\\
&\leq \left(\frac{9}{10}\right)^{2t} n!
\end{align*} 
since $\sum_{\alpha} \dim (R_n(\alpha)) = \left| S_M(\vec{b})\right| \leq n!$. 
Then, using $t$ as defined as in Equation \eqref{tupperboundrestated} above, 
\begin{align*}
Q_1  &\leq \left( \frac{9}{10} \right)^{n^2(\log f(n) + \log g(n))/2f(n)} 
n!\\
& \leq \left( \frac{9}{10} \right)^{n^2(\log f(n) + \log g(n))/2f(n)} e^{n\log n} \\
&\leq \exp\left(n^2\left(\frac{\log n}{n} - \log (10/9) \frac{\log(f(n))}{2f(n)}\right) \right)
\end{align*}
From Lemma \ref{boundcompare}, for sufficiently large $n$, the above is at most $e^{-n \log n} = n^{-n}$, completing the proof. 
\end{proof}

Before continuing, a simple supporting lemma is needed. 

\begin{lemma}\label{aijkbound}
With notation as above, 

\begin{equation*}
\frac{i!k!}{j!} a(i, j, k) \leq \frac{f(n)^i g(n)^{i-j}(n-f(n))^{k-j}(n-g(n))^k}{(i-j)!(k-j)!j!}
\end{equation*}
\end{lemma}
\begin{proof}[\bf Proof:]
This follows easily from the definition of $a(i, j, k)$ in Definition \ref{aijk} and the fact that ${x \choose y} \leq \frac{x^y}{y!}$ .
\end{proof}

\begin{lemma}\label{iiszero}
For 
\begin{equation*}
A_2 = \left\{ (i, j, k) \neq(0,0,0) \left| \right. k < \frac{n-g(n)}{5}, i = 0 \right\}
\end{equation*}  let 
\begin{equation*}
Q_2 = \sum_{(i, j, k) \in A_2} s(i, j, k)^{2t}\sum_{\left| \alpha^{Re} \right| = (i, j, k)} \dim (R_n(\alpha)). 
\end{equation*}
Then, for sufficiently large $n$, $Q_2 \leq 2(n - g(n))^{-6}$. 
\end{lemma}
\begin{proof} [\bf Proof:]
By Definition \ref{sijk}, 
\begin{align*}
s(i, j, k) &= 1 - s_1(i) + s_2(j) - s_3(k) \\
          &\leq e^{-s_1(i) + s_2(j) - s_3(k)} 
\end{align*}
since $1 - x\leq e^{-x}$ for all $x$. Furthermore, from Lemma \ref{dimensionsum},  
\begin{equation*}
\sum_{\left| \alpha^{Re} \right| = (i, j, k)} \dim (R_n(\alpha))  \leq a(i, j, k) \frac{i!k!}{j!}
\end{equation*}
for $a(i,j, k)$ as in Definition \ref{aijk}. Combining the above,  
\begin{equation}\label{mainupperbound}
Q_2 \leq \sum_{(i, j, k) \in A_2} e^{-2ts_1(i) + 2ts_2(j) - 2ts_3(k)} a(i, j, k) \frac{i!k!}{j!}
\end{equation}
Now, for $(i, j, k) \in A_2$, $i = 0$. Since $i \geq j$, this also means that $j = 0$. Using Definition \ref{sijk} and Lemma \ref{aijkbound},
\begin{align*}
s_1(i) = 0, s_2(j) = 0 \text{ and } \frac{i!k!}{j!} a(i, j, k) \leq \frac{(n-f(n))^k (n -g(n))^k}{k!} 
\end{align*}
Thus, the above inequality simplifies to 
\begin{equation*}
Q_2 \leq \sum_{1\leq k \leq \frac{n-g(n)}{5}} \exp\left(-\frac{4tk(n - g(n) - k + 1)}{n+2\Delta}\right) \frac{(n-f(n))^k (n -g(n))^k}{k!} 
\end{equation*}
and hence, using the fact that $g(n) \leq f(n)$, 
\begin{align*}
Q_2 &\leq \sum_{1\leq k < \frac{n-g(n)}{5}} \exp\left(-\frac{4tk(n - g(n) - k + 1)}{n+2\Delta}\right) \frac{(n-g(n))^{2k}}{k!} \\
&\leq \sum_{1\leq k < \frac{n-g(n)}{5}} \exp\left(-\frac{16tk(n - g(n))}{5(n+2\Delta)}\right) \frac{(n-g(n))^{2k}}{k!} 
\end{align*}
Now, using Lemma \ref{boundcompare} with $x_n = f(n)$ and $y_n = n - g(n)$, for sufficiently large $n$, $\frac{t}{n+2\Delta} \geq \frac{f(n)}{4f(n)} \geq 10 \frac{ \log( n - g(n))}{n - g(n)}$. Thus,  
\begin{equation*}
Q_2 \leq \sum_{1\leq k < \frac{n-g(n)}{5}} \exp\left(- 8k\log(n -g(n))\right)\frac{(n-g(n))^{2k}}{k!} 
\end{equation*}
and hence the above simplifies to
\begin{align*}
Q_2  &\leq \sum_{1\leq k < \frac{n-g(n)}{5}} \frac{(n-g(n))^{-6k}}{k!} \leq 2(n - g(n))^{-6}
\end{align*}
as required. 
\end{proof}

It is easy to see that the lemmas above can be manipulated to provide arbitrarily good bounds. As noted previously, this is because the `limiting' term $(1, 0, 0)$ does not make an appearance here. The case $i \neq 0$ will be considerably more tricky (or at least more tedious.) 

\begin{lemma}\label{iisn'tzero}
For
\begin{equation*}
A_3 = \left\{ (i, j, k) \left| \right. k < \frac{n-g(n)}{5}, i \neq 0 \right\}
\end{equation*} 
let 
\begin{equation*}
Q_3 = \sum_{(i, j, k) \in A_2} s(i, j, k)^{2t}\sum_{\left| \alpha^{Re} \right| = (i, j, k)} \dim (R_n(\alpha))  
\end{equation*}
Then, for sufficiently large $n$, 
\begin{equation*}
Q_3 \leq 12e^{-c} + 4e^{-\frac{cf(n)}{10} + f(n)}
\end{equation*}
\end{lemma}
\begin{proof}[\bf Proof:]
Manipulating as in Lemma \ref{iiszero} gives an equation analogous to Equation \eqref{mainupperbound} above: 
\begin{equation*}
Q_3 \leq \sum_{(i, j, k) \in A_3} e^{-2ts_1(i) + 2ts_2(j) - 2ts_3(k)} a(i, j, k) \frac{i!k!}{j!}
\end{equation*}
Do the above sum over $k$, then over $j$, then finally over $i$. Using Lemma \ref{aijkbound}, 
\begin{equation}\label{firstupperbound}
\begin{split}
Q_3 &\leq \sum_{(i, j,k) \in A_3} e^{-2ts_1(i) + 2ts_2(j) - 2ts_3(k)}\frac{f(n)^i g(n)^{i-j}(n-f(n))^{k-j}(n-g(n))^k}{(i-j)!(k-j)!j!}  \\
    & \leq \sum_{(i, j)} e^{-2ts_1(i) + 2ts_2(j)} \frac{f(n)^i g(n)^{i-j}}{(i-j)!j!} \\
&\hspace{0.8 in} \cdot  \sum_{k:(i, j, k) \in A_3} e^{-2ts_3(k)}\frac{(n-f(n))^{k-j}(n-g(n))^k}{(k-j)!}
\end{split}
\end{equation}
Fix $i$ and $j$ and do the sum over $k$:

\paragraph{Summing over $k$:} If $(i, j, k) \in A_3$, then by definition $k \leq (n-g(n))/5$. Also, $k \geq j$, giving
\begin{equation*}
\sum_{k  = j}^{(n-g(n))/5} e^{-2ts_3(k)}\frac{(n-f(n))^{k-j}(n-g(n))^k}{(k-j)!}
\end{equation*}
Denote the $k$th term of the above sum by $r_k$. The idea will be to show that $r_{k+1}/r_k$ is less than $1/2$, and thus to bound the sum by $2r_j$. Explicitly, 
\begin{align*}
\frac{r_{k+1}}{r_k} &= e^{-2ts_3(k+1) +2ts_3(k)}\frac{(n-f(n))(n-g(n))}{k-j+1}\\
                    &\leq \exp\left(\frac{4t(2k - n+g(n))}{n+2\Delta}\right)
									 			\frac{(n-g(n))^2}{k-j+1} \\
                    &\leq \exp\left( \frac{-12t(n - g(n))}{5(n+2\Delta)} \right)(n - g(n))^2
\end{align*}
and again using Lemma \ref{boundcompare} with $x_n = f(n)$ and $y_n = n - g(n)$, for sufficiently large $n$, $\frac{t}{n+2\Delta} \geq \frac{f(n)}{4f(n)} \geq 5 \frac{ \log( n - g(n))}{n - g(n)}$, and thus
\begin{align*}
\frac{r_{k+1}}{r_k} &\leq \exp\left( -12\log(n - g(n)) \right)(n - g(n))^2 \\
 &= (n-g(n))^{-10} \leq \frac{1}{2}
\end{align*}
for sufficiently large $n$. Therefore,  
\begin{align*}
\sum_{ k = j}^{(n-g(n))/5} e^{-2ts_3(k)}\frac{(n-f(n))^{k-j} (n-g(n))^k}{(k-j)!} &= \sum_{ k = j}^{(n-g(n))/5} r_k \leq 2 r_j \\
 &= 2e^{-2ts_3(j)} (n-g(n))^j
\end{align*}
Plugging this back into Equation \eqref{firstupperbound},  
\begin{align}\label{secondupperbound}
Q_3 &\leq \sum_{(i, j)} e^{-2ts_1(i) + 2ts_2(j)} \frac{f(n)^i g(n)^{i-j}}{(i-j)!j!} \left(2e^{-2ts_3(j)} (n-g(n))^j\right) \nonumber \\
    &\leq 2\sum_{i=1}^{f(n)} e^{-2ts_1(i)}f(n)^i \sum_{j = 0}^ie^{2ts_2(j)- 2ts_3(j)} \frac{g(n)^{i-j}(n-g(n))^j}{(i-j)!j!}
\end{align}
again using the fact that if $\left| \alpha^{Re} \right| = (i, j, k)$, then $i \geq j$, and also $i \leq f(n)$. Now fix $i$, and do the sum over $j$. 

\paragraph{Summing over $j$:} The sum to be bounded is
\begin{equation*}
\sum_{j = 0}^ie^{2ts_2(j)- 2ts_3(j)} \frac{g(n)^{i-j}(n-g(n))^j}{(i-j)!j!}
\end{equation*}
From Definition \ref{sijk}, 
\begin{equation*}
s_2(j) - s_3(j) = - \frac{2j(n - f(n))}{n+2\Delta}
\end{equation*}
Hence the sum simplifies to
\begin{align*}
\sum_{j =0}^i \exp \left( -\frac{4tj (n - f(n))}{n + 2\Delta}\right)  \frac{g(n)^{i-j}(n-g(n))^j}{(i-j)!j!}
\end{align*}
By a slight abuse of notation, let $r_j$ again be the $j$th summand of the above sum. As above, bound the ratio between $r_{j+1}$ and $r_j$ to bound the sum. Here, 
\begin{align*}
\frac{r_{j+1}}{r_j} &= \exp\left(- \frac{4t(n-f(n))}{n+2\Delta} \right) \frac{(n-g(n))(i-j)}{g(n)(j+1)} \\
 &\leq\exp\left(- \frac{4t(n-f(n))}{n+2\Delta} \right) (n - g(n))f(n)
\end{align*}
Now, using Lemma \ref{boundcompare} with $x_n = f(n)$ and $y_n = n - f(n)$, for sufficiently large $n$, $\frac{t}{n+2\Delta} \geq \frac{\log f(n)}{4f(n)} \geq \frac{\log (n - f(n))}{n - f(n)}$. Thus,  
\begin{align*}
\frac{r_{j+1}}{r_j} &= \exp\left(- 4\log(n - f(n)) \right) (n - g(n))f(n) \\
 &\leq \frac{(n - g(n))f(n)}{(n - f(n))^4} < \frac{1}{2}
\end{align*}
for sufficiently large $n$, and as before, 
\begin{align*}
\sum_{j = 0}^ie^{2ts_2(j)- 2ts_3(j)} \frac{g(n)^{i-j}(n-g(n))^j}{(i-j)!j!} =  \sum_{j = 0}^i r_j \leq 2r_0 = \frac{2g(n)^i}{i!}
\end{align*}
Plugging this back into Equation \eqref{secondupperbound},  
\begin{equation}\label{thirdupperbound}
Q_3 \leq 4\sum_{i=1}^{f(n)} e^{-2ts_1(i)}\frac{f(n)^ig(n)^i}{i!}. 
\end{equation}

\paragraph{Summing over $i$:} Break up the above sum into two pieces: $ i \leq \frac{f(n)}{5}$ and $i > \frac{f(n)}{5}$. Bound the first case first. Consider the summation 
\begin{equation*}
\sum_{i=1}^{f(n)/5} e^{-2ts_1(i)}\frac{f(n)^ig(n)^i}{i!}
\end{equation*}
Recall that  
\begin{equation}\label{sijkrestated}
s_1(i) = \begin{cases}
   \frac{2i(f(n) - i+1)}{n+2\Delta} & i < \frac{f(n)}{2}\\
   \frac{if(n)}{n+2\Delta} & i\geq \frac{f(n)}{2}
\end{cases}
\end{equation}
Thus, for $i < \frac{f(n)}{5}$, 
\begin{align*}
2ts_1(i) &= 2\left(\frac{(n+2\Delta)(\log f(n) + \log g(n))}{4f(n)} + c \frac{n+2\Delta}{4f(n)}\right)\frac{2i(f(n) - i+1)}{n+2\Delta} \\
         & = (\log f(n) + \log g(n) + c) \left( i - \frac{i^2 - i}{f(n)}\right)
\end{align*}
Now, $i - \frac{i^2 - i}{f(n)}$ is a quadratic function in $i$ which corresponds to an upside down parabola, and as such is minimized at the endpoints of an interval. Plugging in $i = 1$ and $i = \frac{f(n)}{5}$,  for $1 \leq i \leq \frac{f(n)}{5}$, $i - \frac{i^2 - i}{f(n)} \geq 1$. This gives that
\begin{equation*}
2t s_1(i) \geq (\log f(n)+ \log g(n)) \left( i - \frac{i^2 - i}{f(n)}\right) + c
\end{equation*}
Thus, 
\begin{align*}
e^{-2ts_1(i)} &\leq \exp \left(-(\log f(n)+ \log g(n)) \left( i - \frac{i^2 - i}{f(n)}\right) - c \right)\\
 &= f(n)^{-i} g(n)^{-i} f(n)^{\frac{i^2 - i}{f(n)}} g(n)^{\frac{i^2 - i}{f(n)}} e^{-c}
\end{align*}
Hence, 
\begin{equation*}
\sum_{i=1}^{f(n)/5} e^{-2ts_1(i)}\frac{f(n)^ig(n)^i}{i!} \leq e^{-c} \sum_{i=1}^{f(n)/5} f(n)^{\frac{i^2 - i}{f(n)}} g(n)^{\frac{i^2 - i}{f(n)}}\frac{1}{i!}
\end{equation*}
By yet another slight abuse of notation, let $r_i$ be the $i$th summand of the above right-hand sum -- bound the ratio between consecutive terms to find bounds on the sum. Then, since $g(n) \leq f(n)$, 
\begin{align*}
\frac{r_{i+1}}{r_i} &= \frac{1}{i+1} f(n)^{\frac{2i}{f(n)}}g(n)^{\frac{2i}{f(n)}} \leq \frac{1}{i+1} f(n)^{\frac{4i}{f(n)}}
\end{align*}

Now, consider the above expression on the right-hand side as a function of $i$. By differentiating, it is easy to check that it is increasing for $i+1 > \frac{f(n)}{4 \log f(n)}$ and decreasing for $i+1 < \frac{f(n)}{4 \log f(n)}$. Thus, in order to find its maximum, just check the endpoints. Plugging $i = 1$, the result is $\frac{1}{2} f(n)^{4/f(n)}$. Since $f(n) \rightarrow \infty$, this approaches $1/2$ as $n \rightarrow \infty$. Thus, for sufficiently large $n$, this is at most $2/3$. Plugging in $i = f(n)/5$, 
\begin{equation*}
\frac{1}{f(n)/5+1} f(n)^{\frac{4f(n)/5}{f(n)}} \leq 5f(n)^{-1/5}
\end{equation*}
which clearly goes to $0$, and hence is less than $2/3$ for sufficiently large $n$. Therefore, $r_{i+1}/r_i$ is less than $2/3$ for sufficiently large $n$, and hence 
\begin{equation*}\
\sum_{i=1}^{f(n)/5} f(n)^{\frac{i^2 - i}{f(n)}} g(n)^{\frac{i^2 - i}{f(n)}}\frac{1}{i!} = \sum_{i=1}^{f(n)/5} r_i \leq 3r_1 = 3
\end{equation*}
and therefore
\begin{equation}\label{ismall}
\sum_{i=1}^{f(n)/5} e^{-2ts_1(i)}\frac{f(n)^ig(n)^i}{i!} \leq 3e^{-c}
\end{equation}
It remains to bound the sum for $i\geq f(n)/5$. This is 
\begin{equation*}
\sum_{i=f(n)/5+1}^{f(n)} e^{-2ts_1(i)}\frac{f(n)^ig(n)^i}{i!} 
\end{equation*}
It is easy to check from the restated definition in Equation \eqref{sijkrestated}  that for $i > \frac{f(n)}{5}$, 
\begin{equation*}
s_1(i) \geq \frac{if(n)}{n+2\Delta}
\end{equation*}
Thus, 
\begin{align*}
e^{-2ts_1(i)}& \leq \exp\left(-2\left(\frac{(n+2\Delta)(\log f(n) + \log g(n))}{4f(n)} + c \frac{(n+2\Delta)}{4f(n)}\right) \frac{if(n)}{n+2\Delta} \right)\\
&=\exp \left( - \frac{i(\log f(n) + \log g(n)+c)}{2}\right)\\
              &= e^{-\frac{ic}{2}} f(n)^{-\frac{i}{2}} g(n)^{-\frac{i}{2}}
\end{align*}
Therefore, since $g(n)\leq f(n)$,
\begin{align*}
\sum_{i=f(n)/5+1}^{f(n)} e^{-2ts_1(i)}\frac{f(n)^ig(n)^i}{i!}  &\leq \sum_{i = f(n)/5+1}^{f(n)} e^{-\frac{ic}{2}} \frac{f(n)^{\frac{i}{2}} g(n)^{\frac{i}{2}}}{i!} \\
&\leq e^{-\frac{cf(n)}{10}} \sum_{i=f(n)/5+1}^{f(n)} \frac{f(n)^i}{i!} \\
&\leq e^{-\frac{cf(n)}{10} + f(n)} 
\end{align*}
Combining the above with Equation \eqref{ismall}, 
\begin{equation*}
\sum_{i = 1}^{f(n)} e^{-2ms_1(i)}\frac{f(n)^i g(n)^i}{i!} \leq 3e^{-c} + e^{-\frac{cf(n)}{10} + f(n)} 
\end{equation*}
and thus from Equation \eqref{thirdupperbound}, 
\begin{equation*}
Q_3 \leq 12e^{-c} + 4e^{-\frac{cf(n)}{10} + f(n)} 
\end{equation*}
as required. 
 \end{proof}

\begin{proof}[\bf Proof of Upper Bound in Theorem \ref{chitheorem}]
Combining Equation \eqref{sumtobound} with Lemmas \ref{klarge}, \ref{iiszero} and \ref{iisn'tzero}, 
\begin{align*}
\sum_{1 \neq \Lambda_1(\alpha) \geq 0} \dim(R_n(\alpha)) \Lambda_1(\alpha)^{2t}   
 &\leq \sum_{(i, j, k) \neq (0, 0, 0)} s(i, j, k)^{2t}\sum_{\left| \alpha^{Re} \right| = (i, j, k)} \dim (R_n(\alpha)) \\
 &\leq  n^{-n} +  2(n - g(n))^{-6} + 12e^{-c} + 4e^{-\frac{cf(n)}{10} + f(n)} 
\end{align*}
Thus, for $c > 10$ and $n$ sufficiently large,  
\begin{align*}
\sum_{1 \neq \Lambda_1(\alpha) \geq 0} \dim(R_n(\alpha)) \Lambda_1(\alpha)^{2t}   
 &\leq   16e^{-c}
\end{align*}
Now, from Equation \eqref{alleigsequation}, 
\begin{align*}
 \left\| P^t(x, \cdot) - \pi \right\|_{2, \pi} = \sqrt{\sum_{\alpha} \dim(R_n(\alpha)) \Lambda_1(\alpha)^{2t}}
\end{align*}
and thus, for $c > 10$ and for $n$ sufficiently large, 
\begin{equation*}
 \left\| P^t(x, \cdot) - \pi \right\|_{2, \pi} \leq 4e^{-\frac{c}{2}}
\end{equation*}
as required. 
\end{proof}
\section{Open Questions} \label{sectionquestionstoponder}

As noted in the introduction, there are many questions raised by this paper which could be profitably explored. The most approachable one concerns the conditions in Theorem \ref{chitheorem}, which assumes that 
\begin{equation*}
\lim_{n\rightarrow \infty}\frac{f(n)}{n} = 0 \textnormal{ and } \lim_{n\rightarrow \infty} f(n) = \infty
\end{equation*}
It is not difficult to show that the second condition about $f(n)$ approaching infinity is necessary (this fact is first noted in Remark \ref{whyfgoestoinfinity}). Indeed, as Section \ref{sectheuristics} demonstrates, the first $f(n)$ rows are the limiting factors behind the mixing time. Furthermore, the highest eigenvalue calculations in Section \ref{sectchisquaredlower} also show that $f(n)$ drives the highest eigenvalue. For cutoff to occur, the driving force behind mixing needs to grow without bound, explaining why the condition is necessary. 

However, the assumption that $\frac{f(n)}{n}$ approaches $0$ does not have a similarly natural justification. This assumption was largely made to make the calculations more tractable. Indeed, tracing back through the proof shows that an identical argument could be made using only the assumption that 
\begin{equation*}
\limsup_{n \rightarrow \infty} \frac{f(n)}{n} \leq c
\end{equation*}
where $c$ is some specific constant. While I haven't calculate precisely what it would need to be, something on the order of $0.1$ would suffice. 

Conversely, it's clear from the heuristics and the eigenvalue calculations that in order for the logic in this paper to be valid, $f(n)$ needs to be no bigger than $n - g(n)$. Otherwise, the highest eigenvalue is the one associated to $n - g(n)$, and the mixing is driven by the last $n - g(n)$ columns instead of the first $f(n)$ rows. Thus, the condition that in the limit, 
\begin{equation*}
f(n) + g(n) \leq n 
\end{equation*} 
is necessary. This leads to the first question:

\begin{quest}\label{conditionstheoremchi}
How much can the conditions in Theorem \ref{chitheorem} be relaxed? If the assumption is that 
\begin{equation*}
\limsup_{n \rightarrow \infty} \frac{f(n)}{n} \leq c
\end{equation*}
what value of $c$ would allow a virtually identical proof to go through? What are the strongest conditions on $f(n)$ (or possibly $f(n)$ and $g(n)$ both) that still allow for the same chi-squared cutoff?
\end{quest}

Continuing to pose questions, part of the focus of this paper is on the discrepancy between chi-squared and total variation mixing for two-step restriction matrices. Indeed, Theorem \ref{TVlowerbound} states that as long as $f(n)$ and $g(n)$ are commeasurable in the limit, then the two mixing times coincide (and both exhibit cutoff); whereas Theorem \ref{TVfastmixingexample} states that for a wide class of $f(n)$, as long as $g(n) = 1$, total variation mixing occurs substantially earlier and without cutoff. The same theorem would also go through without major adjustments under the assumptions that $g(n)\leq c$ for any constant $c$; however, this is unlikely to be the only case where something similar happens. This leads to a number of natural and interesting questions:

\begin{quest}
What conditions on $f(n)$ and $g(n)$ are necessary so that total variation mixing occurs substantially before chi-squared mixing? Are those precisely the cases when the total variation mixing time doesn't undergo cutoff? If those are distinct phenomena, under which conditions on $f(n)$ and $g(n)$ does the random walk undergo total variation cutoff?
\end{quest}

A considerably more ambitious avenue of research involves studying a broader class of one-sided restriction matrices. In this paper, attention was focused on two-step restriction matrices because they induce a vertex transitive random transposition walk. However, I would conjecture that very few other restriction matrices have this property; see Remark \ref{examplenonvertextransitive} for a demonstration of what can happen. (For an example of a one-sided restriction matrix which isn't in this class but does induce a vertex transitive walk, see Example 3 in Hanlon \cite{HanlonPaper}.) This assumption made possible the use of Corollary \ref{alleigsvertextransitive}, which requires only eigenvalue and not eigenvector information to achieve the requisite chi-squared bounds. 

However, Hanlon does derive both eigenvalue and eigenvector information in his paper. While the calculations would be more difficult, it would certainly be feasible to apply Theorem \ref{alleigs} in the case of a general one-sided restriction matrix. This means that it would be possible to tackle questions similar to the ones in this paper for many more types of one-sided restriction matrices. Some of the questions which arise naturally are the following:

\begin{quest}
Is there a wider class of one-sided restriction matrices for which the random transposition walk can be proven to undergo chi-squared cutoff, and where does this cutoff occur? What kind of conditions on the restriction matrix ensure that the chi-squared and total variation mixing times match? Is there a wide class of examples for which the total variation time is an order smaller than the chi-squared mixing time? 
\end{quest} 

Although this brief summary is a good start, many analogous questions on the random transposition walk on one-sided restriction matrices can be posed.

    \bibliographystyle{plain}

    \bibliography{mybib}
\end{document}